\documentclass[
a4paper 
,11pt 
,leqno 
]{scrartcl} 

\usepackage{my-preamble-article} 
\usepackage{my-macros} 
\usepackage{tikzsymbols} 

\usepackage{my-specific-macros} 

\title{Competition between stable equilibria in reaction-diffusion systems: the influence of mobility on dominance}

\author{Emmanuel \textsc{Risler}\footnote{\protect\url{\myWebPage}}}
\affil{\myAffiliation}

\begin{document}

\maketitle
\thispagestyle{empty}

\begin{abstract}
This paper is concerned with reaction-diffusion systems of two symmetric species in spatial dimension one, having two stable symmetric equilibria connected by a symmetric standing front. The first order variation of the speed of this front when the symmetry is broken through a small perturbation of the diffusion coefficients is computed. This elementary computation relates to the question, arising from population dynamics, of the influence of mobility on dominance, in reaction-diffusion systems modelling the interaction of two competing species. It is applied to two examples. First a toy example, where it is shown that, depending on the value of a parameter, an increase of the mobility of one of the species may be either advantageous or disadvantageous for this species. Then the Lotka--Volterra competition model, in the bistable regime close to the onset of bistability, where it is shown that an increase of mobility is advantageous. Geometric interpretations of these results are given. 
\end{abstract}

\section{Introduction}
Reaction-diffusion systems play an important role as models for a large variety of spatio-temporal systems arising from various fields: Chemistry, Physics, Mechanics, Genetics, Ecology\dots. An relevant concept for the understanding of their dynamical behaviour is the dominance of equilibria, \cite{Fife_reactDiffSyst_1979}: given two stable homogeneous equilibria, in which sense can one say that an equilibrium ``dominates'' the other one~? A possible answer is (see for instance \cite{HutsonVickers_travWaveDom_1992}): equilibrium $A$ dominates equilibrium $B$ if there exists a travelling front connecting these two equilibria and displaying invasion of $B$ by $A$ (even if the order relation induced by this definition is not always antisymmetric, see the example in appendix, \vref{subsec:ex_2_eq}). 

A natural related question is that of the influence of mobility on dominance: how is the speed of a front connecting two stable equilibria (and in particular the sign of this speed) affected by a change in the diffusion coefficients~? This question is relevant in the context of population dynamics. Consider a system modelling the evolution of densities of two species competing in a one-dimensional environment. In this case one expects the existence of two stable equilibria, each corresponding to the dominance of a species, for the local reaction system. The question above is that of the influence of the mobility of each of the two species on their relative fitness, that is on the sign of the speed of a front connecting these equilibria. 

One may believe that less mobility is always advantageous, since it reduces the dispersal at the interface where the two species coexist, and thus prevents invasion (see the observations in \cite{HutsonVickers_travWaveDom_1992}). But other effects can be invoked: an increase in the mobility of, say, the first species, changes the total density of individuals on each side of the interface, and results in undercrowding on the side where first species dominates and overcrowding on the other side, an effect having unclear consequences. While according to results of A. Hastings \cite{Hastings_canSpatialVariationAloneLeadToSelectionForDispersal_1983} and J. Dockery et al. \cite{DockeryHutson_evolutionSlowDispersalRates_1998} a heterogeneous environment seems to be always in favour of a reduction of dispersal, V. Hutson and G. T. Vickers made on a model the numerical observation that large or small diffusion cannot unambiguously be claimed to be favourable in general, \cite{HutsonVickers_titForTatDefect_1995}. More recently, L. Girardin and G. Nadin considered a Lotka--Volterra competition model close to the infinite competition limit and proved in this case a ``Unity is not strength'' result stating that a large dispersal is favoured, \cite{GirardinNadin_tWRelativeMotilityInvasionSpeed_2015}.

The aim of this paper is to examine on some cases the value of the first order dependence of the speed of a bistable front with respect to a perturbation of the diffusion matrix, and to try to determine the sign of this quantity. First we consider a general reaction-diffusion system in spatial dimension one, governing two symmetric scalar components, and assume the existence of two stable homogeneous equilibria that are symmetric (with respect to exchange of the two components) and connected by a symmetric (thus stationary) front. Then the symmetry between the two scalar components is broken by a small perturbation of the diffusion matrix (say a small increase of the diffusion coefficient of the first component) and several expressions are provided for the first order dependence of the speed of the front with respect to this perturbation (\cref{sec:general}), together with a geometric interpretation for some of these expressions. All this suggests that both signs may occur for this first order dependence, depending on the features of the initial system.

Two specific examples are then considered. First (\cref{sec:toy}) a toy example where the initial standing front is explicit, and where it is shown that both signs (for the first order dependence introduced above) actually occur, depending on the value of a parameter of the system. This confirms on a computable case the aforementioned observations of Hutson and Vickers. The second example (treated in \cref{sec:lv}) is the Lotka--Volterra competition model in the bistable regime, close to the onset of bistability. Using singular perturbation arguments, it is shown in this case that a large dispersal is advantageous. 
\section{Assumptions, notation, perturbation scheme}
\label{sec:general}
\subsection{Setup}
\label{subsec:setup}
Let us consider the reaction-diffusion system:
\begin{equation}
\label{react_diff}
u_t = F(u)+\ddd u_{xx}
\end{equation}
where the time variable $t$ and the space variable $x$ are real, space domain is the full real line, the field variable $u$ is $n$-dimensional ($n$ is a positive integer), the ``reaction'' function $F:\rr^n\rightarrow\rr^n$ is smooth, and the ``diffusion'' matrix $\ddd$ is a positive definite symmetric $n\times n$ real matrix. 
Let us assume that this system admits two distinct spatially homogeneous equilibria, in other words that there exist two points $E_-$ and $E_+$ in $\rr^n$ such that 
\[
E_-\not= E_+ 
\quad\mbox{and}\quad
F(E_-)=F(E_+)=0
\]
and let us assume that there exists a travelling front connecting these two equilibria, in other words that there exist a smooth function $\phi:\rr\rightarrow\rr^n$ and a real quantity $c$ such that the function $(x,t)\mapsto\phi(x-c t)$ is a solution of system \cref{react_diff} and such that
\[
\phi(\xi)\rightarrow E_-
\quad\mbox{when}\quad
\xi\rightarrow-\infty
\quad\mbox{and}\quad
\phi(\xi)\rightarrow E_+
\quad\mbox{when}\quad
\xi\rightarrow+\infty
\,.
\]
This function $\phi$ is a global solution of the system:
\begin{equation}
\label{syst_front}
-c\phi'(\xi)=F\bigl(\phi(\xi)\bigr)+\ddd\phi''(\xi)
\,.
\end{equation}
Let us assume in addition that both equilibria $E_-$ and $E_+$ are hyperbolic (that is the linear functions $DF_{E_+}$ and $DF_{E_-}$ have no eigenvalue with zero real part). In this case both functions $\xi\mapsto F\bigl(\phi(\xi)\bigr)$ and $\xi\mapsto \phi'(\xi)$ approach $0_{\rr^n}$ at an exponential rate when $\xi$ approaches $\pm\infty$, and as a consequence these functions belong to the space $L^2(\rr,\rr^n)$. 
Let us denote by `` $\cdot$ '' the canonical scalar product in $\rr^n$, and let
$\langle \cdot , \cdot \rangle_{L^2(\rr,\rr^n)}$ and $\norm{\cdot}_{L^2(\rr,\rr^n)}$ denote the usual scalar product and corresponding norm in $L^2(\rr,\rr^n)$, namely (for every pair $(f,g)$ of functions of $L^2(\rr,\rr^n)$):
\[
\langle f , g \rangle_{L^2(\rr,\rr^n)} = \int_{-\infty}^{+\infty} f(x) \cdot g(x) \, dx
\quad\mbox{and}\quad
\norm{f}_{L^2(\rr,\rr^n)} = \sqrt{\langle f , f \rangle_{L^2(\rr,\rr^n)}}
\,.
\]
Now, it follows from system \cref{syst_front} that the quantity $c$ admits the following explicit expression:
\begin{equation}
\label{expr_c0}
c 
= -\frac{\int_{-\infty}^{+\infty} F\bigl(\phi(\xi)\bigr)\cdot\phi'(\xi)\,d\xi} {\int_{-\infty}^{+\infty}\phi'^2(\xi)\,d\xi}
= -\frac{\langle F(\phi) , \phi' \rangle_{L^2(\rr,\rr^n)}}{\norm{\phi'}_{L^2(\rr,\rr^n)}^2}
\,.
\end{equation}

If the reaction function $F$ derives from a potential $V:\rr^n\rightarrow\rr$ (namely if $F(u)=-\nabla V (u)$ for all $u$ in $\rr^n$) then this expression of $c$ becomes:
\begin{equation}
\label{expr_c0_pot}
c=\frac{V(E_+)-V(E_-)}{\norm{\phi'}_{L^2(\rr,\rr^n)}^2}
\,. 
\end{equation}
Thus, in this case, the sign of the speed $c$ of the front only depends on the sign of the difference between $V(E_+)$ and $V(E_-)$. In particular, if there exist several 
travelling fronts connecting $E_-$ to $E_+$, then all the velocities of these fronts have the same sign. Such is not always the case when $F$ does not derive from a potential: it is not difficult to construct an example of system of the form \cref{react_diff} where two distinct equilibria are connected by two travelling fronts with velocities of opposite signs (for sake of completeness such an example is given in appendix, see \vref{subsec:ex_2_eq}).

In the following we shall not assume that $F$ derives from a potential. Our aim is to understand the influence of a small change in the diffusion matrix $\ddd$ on the speed $c$ of the travelling front $\phi$. 
\subsection{Stability and transversality assumptions}
\label{subsec:transv_assump}
Let us introduce the space coordinate $\xi=x-ct$ in a frame travelling at speed $c$. If two functions $u(x,t)$ and $v(\xi,t)$ are related by:
\[
u(x,t) = v(\xi,t)= v(x-ct,t)
\,,
\]
then $u$ is a solution of system \cref{react_diff} if and only if $v$ is a solution of:
\begin{equation}
\label{react_diff_trav_frame}
v_t = c v_\xi + F(v) + \ddd v_{\xi\xi}
\,,
\end{equation}
which represents system \cref{react_diff} rewritten in the $(\xi,t)$ coordinates system. The profile $\xi\mapsto \phi(\xi)$ of the travelling front considered in \cref{subsec:setup} is a steady state of system \cref{react_diff_trav_frame}.
A small perturbation 
\[
(\xi,t)\mapsto \phi(\xi) + \varepsilon v(\xi,t)
\]
of the profile of the front is (at first order in $\varepsilon$) a solution of \cref{react_diff_trav_frame} if and only if $v$ is a solution of the linearised system:
\begin{equation}
\label{react_diff_trav_frame_lin}
v_t = c v_\xi + DF(\phi) v + \ddd v_{\xi\xi}
\,.
\end{equation}
The right-hand side of \cref{react_diff_trav_frame_lin} defines the differential operator
\begin{equation}
\label{def_linearised_operator}
\mathcal{L}:c\partial_\xi+ DF(\phi) + \ddd\partial_{\xi\xi}
\,.
\end{equation}
Considered as an unbounded operator in $L^2(\rr,\rr^n)$, it is a closed operator with dense domain $H^2(\rr,\rr^n)$. 
Due to translation invariance in the space variable $x$, zero is an eigenvalue of this operator; indeed, differentiating system~\cref{syst_front} yields: 
\[
\mathcal{L}\phi'=0
\,.
\]
Let us make the following hypotheses. 
\begin{description}
\item[\hypStabInfty] The spatially homogeneous equilibria $E_-$ and $E_+$ at both ends of the front are spectrally stable for the reaction-diffusion system \cref{react_diff}.
\end{description}
In other words, For every real quantity $k$, all eigenvalues of the $n\times n$ real matrices 
\[
DF(E_-)-k^2\ddd
\quad\mbox{and}\quad
DF(E_+)-k^2\ddd
\]
have negative real parts (the subscript ``stab-ends'' refers to: ``stable at both ends of space''). Equivalently, the essential spectrum of operator $\mathcal{L}$ is stable \cite{Henry_geomSemilinParab_1981,Sandstede_stabilityTW_2002}. 
\begin{description}
\item[\hypTransv] The eigenvalue zero of the operator $\mathcal{L}$ has an algebraic multiplicity equal to $1$.
\end{description}
In other words, the kernel of operator $\mathcal{L}$ is reduced to $\spanset(\phi')$, and the function $\phi'$ does not belong to $\imm(\mathcal{L})$. The subscript ``transv'' refers to: ``transverse''; indeed, this hypothesis is equivalent to the transversality of the travelling front (see \vref{lem:alg_mult_transv}). 

The two next definitions call upon a topology on the space of travelling fronts, which may be chosen as follows: two travelling fronts $\phi_1$ and $\phi_2$ travelling at speeds $c_1$ and $c_2$ are close if: there exists a translate of $\phi_2$ that is close to $\phi_1$ (uniformly on $\rr$), and the two speeds $c_1$ and $c_2$ are close. 
\begin{definition}[isolation and robustness of the travelling front]
The travelling front $\phi$ is said to be \emph{isolated} if there exists a neighbourhood of it such that every other travelling front of the same system \cref{react_diff} in this neighbourhood is equal to a translate of $\phi$ (and as a consequence travels at the same speed). 

The travelling front $\phi$ is said to be \emph{robust} if every sufficiently small perturbation of system \cref{react_diff} possesses a unique (up to space translation) front close to $\phi$ and travelling with a speed close to $c$. 
\end{definition}
The following statement is a rather standard transversality result \cite{CoulletRieraTresser_stableStaticLocStructOneDim_2000,Coullet_locPattFronts_2002,
Sandstede_stabilityTW_2002,HomburgSandstede_homocHeteroclinicBifVectFields_2010,
GuckenheimerKrauskopf_invManifGlobalBif_2015}. 
\begin{proposition}[isolation and robustness of $\phi$]
\label{prop:robustness}
It follows from hypotheses \hypStabInfty and \hypTransv that the travelling front $\phi$ under consideration is isolated and robust. 
\end{proposition}
For sake of completeness a proof of this proposition is provided in \vref{subsec:robustness}.
\subsection{Spectral stability}
\label{subsec:spec_stab_gen}
The travelling front $\phi$ is said to be
\emph{spectrally stable} if hypotheses \hypStabInfty and \hypTransv are satisfied, and if moreover every nonzero eigenvalue of $\mathcal{L}$ has a negative real part.
In this case the travelling front is then also \emph{non linearly stable with asymptotic phase}, that is for every function $u_0:\rr\rightarrow\rr^n$ sufficiently close (say uniformly on $\rr$) to a translate of $\phi$, there exists a real quantity $x_1$ such that the solution of system \cref{react_diff} with initial condition $u_0$ approaches the solution $(x,t)\mapsto\phi(x-x_1-c t)$ (at an exponential rate) when $t$ approaches $+\infty$ \cite{Henry_geomSemilinParab_1981,Sandstede_stabilityTW_2002}. 

In the two practical examples that will be considered in \cref{sec:toy} and \cref{sec:lv}, the fronts under consideration will be spectrally stable indeed. However, we shall not make any additional spectral stability hypothesis at this stage since such an hypothesis is not required for the general considerations that will be made in the next \cref{subsec:solvency,subsec:altern_expr_bar_c,subsec:with_symmetries,subsec:geom_interpret_bar_c,subsec:red_sym}. 
\subsection{Kernel of the adjoint linearised operator}
\label{subsec:ker_adjoint}
Let $\mathcal{L}^*$ denote the adjoint operator of $\mathcal{L}$ for the scalar product $\langle . , .\rangle_{L^2(\rr,\rr^n)}$, that is:
\[
\mathcal{L}^* = -c\partial_\xi+DF(\phi)^* + \ddd\partial_{\xi\xi}
\]
(see \cref{fig:notation_operators}).
\begin{figure}[!htbp]
	\centering
    \includegraphics[width=0.85\textwidth]{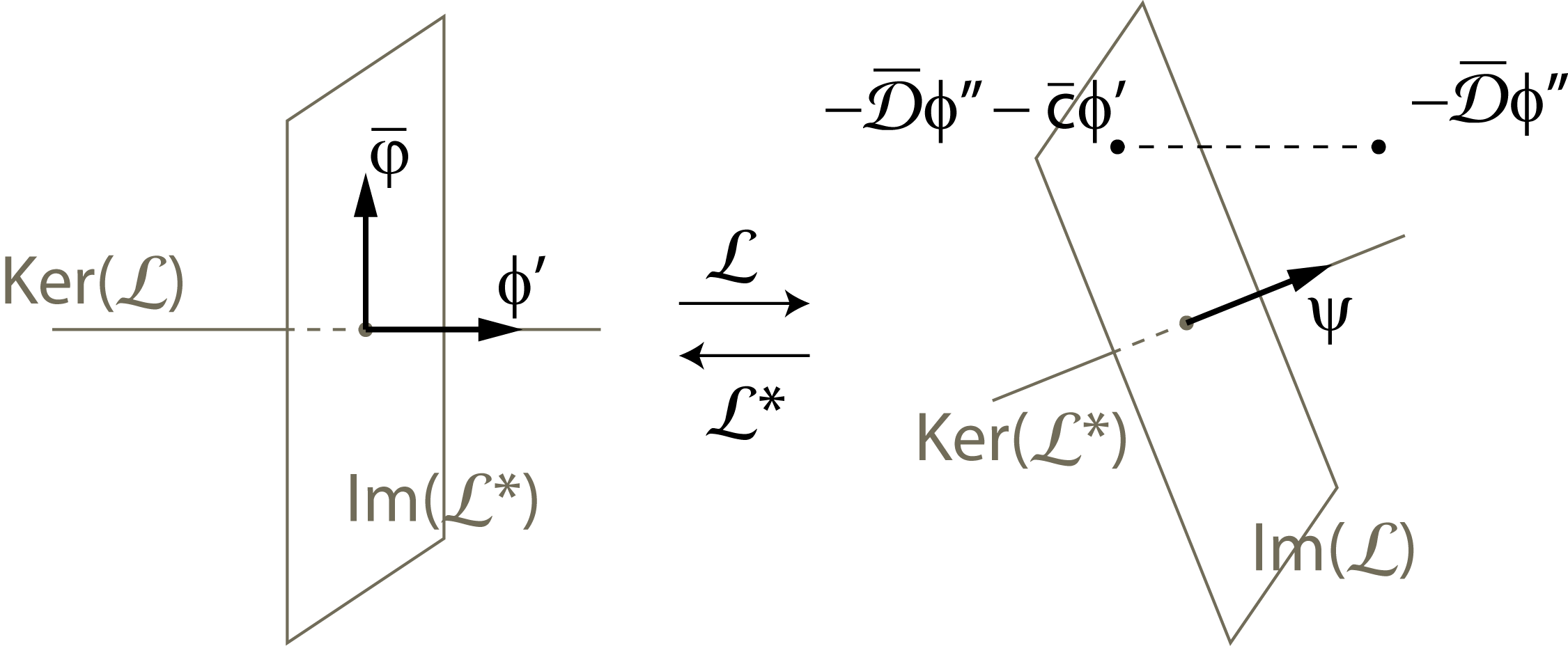}
    \caption{Notation related to the operators $\mathcal{L}$ and $\mathcal{L}^*$.}
    \label{fig:notation_operators}
\end{figure}
Hypotheses \hypStabInfty and \hypTransv ensure that $\ker(\mathcal{L}^*)$ is also one-dimensional \cite{SandstedeScheel_stabTWLargeSpatialPeriod_2001,Sandstede_stabilityTW_2002}, and according to \hypTransv the subspaces $\ker(\mathcal{L})$ and $\ker(\mathcal{L}^*)$ are not orthogonal to one another. As a consequence there exists a unique function $\psi$ in $\ker(\mathcal{L}^*)$, satisfying the normalization condition
\begin{equation}
\label{norm_cond}
\langle\psi,\phi'\rangle_{L^2(\rr,\rr^n)}=1
\,,
\end{equation}
and $\psi(x)$ approaches $0_{\rr^n}$ at an exponential rate when $x$ approaches $\pm\infty$ \cite{SandstedeScheel_stabTWLargeSpatialPeriod_2001,Sandstede_stabilityTW_2002}. 
\subsection{Perturbation of the diffusion matrix and solvency condition}
\label{subsec:solvency}
Let us consider a symmetric (not necessarily positive definite) $n\times n$ real matrix $\bar{\ddd}$, a small positive quantity $\epsilon$, and the following perturbation of system \cref{react_diff}: 
\begin{equation}
\label{r_diff_pert}
u_t=F(u)+(\ddd+\epsilon\bar{\ddd})u_{xx}
\,, 
\end{equation}
According to the consequences of hypotheses \hypStabInfty and \hypTransv mentioned in \cref{subsec:transv_assump}, if $\epsilon$ is sufficiently small, then the perturbed system~\cref{r_diff_pert} admits a unique travelling front close to $\phi$, having a speed close to $c$, and those depend smoothly on $\epsilon$. 
If we denote by $\phi+\epsilon\varphi$ this travelling front and by $c+\epsilon \bar{c}$ its speed, then, replacing these two ansatzes into system \cref{r_diff_pert}, we find that, at first order in $\epsilon$, the function $\varphi$ and the quantity $\bar{c}$ must satisfy the system
\begin{equation}
\label{1st_ord_front}
\mathcal{L}\varphi=-\bar{\ddd}\phi''- \bar{c} \phi'
\,. 
\end{equation}
Taking on both sides the scalar product by $\psi$, it follows that:
\begin{equation}
\label{solv_cond}
\bar c=-\langle\psi,\bar{\ddd}\phi''\rangle_{L^2(\rr,\rr^n)}
\,.
\end{equation}
This is a solvency condition that ensures that $-\bar{\ddd}\phi''-\bar c\phi'$ is orthogonal to the kernel of $\mathcal{L}^*$, thus equivalently that it belongs to the image of $\mathcal{L}$ (see \cref{fig:notation_operators}). 

The main purpose of this paper is to investigate the sign of the quantity $\bar{c}$, since it is this sign that determines how the perturbation in \cref{r_diff_pert} balances the relative dominance of the two equilibria $E_-$ and $E_+$, through the speed of the travelling front $\phi$. Indeed, 
\begin{itemize}
\item if $\bar{c}$ is positive, then, for $\epsilon$ small positive, the influence of the perturbation will be to increase the speed of the front, thus to promote $E_-$ with respect to $E_+$;
\item if conversely $\bar{c}$ is negative, then again for $\epsilon$ small positive, the influence of the perturbation will be to decrease the speed of the front, thus to promote $E_+$ with respect to $E_-$.
\end{itemize}
\subsection{Alternative expression for the first order variation of the speed}
\label{subsec:altern_expr_bar_c}
We are now going to provide a second expression of $\bar{c}$ that will turn out to be useful, and in particular easier to interpret than the solvency condition \cref{solv_cond}.
Since the function $-\bar{\ddd}\phi''-\bar c\phi'$ belongs to the image of $\mathcal{L}$, system \cref{1st_ord_front} admits exactly one solution $x\mapsto\bar\varphi(x)$ satisfying 
\[
\langle \phi',\bar\varphi\rangle_{L^2(\rr,\rr^n)} = 0
\]
(see \cref{fig:notation_operators}).
Taking the scalar product by $\phi'$ in system \cref{1st_ord_front} and integrating over $\rr$, we get
\[
-\bar c\norm{\phi'}_{L^2(\rr,\rr^n)}^2 
= \langle\phi',\mathcal{L}\bar\varphi\rangle_{L^2(\rr,\rr^n)}
= \langle\mathcal{L}^*\phi',\bar\varphi\rangle_{L^2(\rr,\rr^n)}
\,,
\]
and since $\mathcal{L}\phi'=0$, the following alternative expression for $\bar c$ follows:
\begin{equation}
\label{solv_cond_alt}
\bar c 
= \frac{\bigl\langle(\mathcal{L}-\mathcal{L}^*)\phi',\bar\varphi\bigr\rangle_{L^2(\rr,\rr^n)}}{\norm{\phi'}_{L^2(\rr,\rr^n)}^2}
= \frac{\bigl\langle\bigl(DF(\phi)-DF(\phi)^*\bigr)\phi'+2 c \phi'',\bar\varphi\bigr\rangle_{L^2(\rr,\rr^n)}}{\norm{\phi'}_{L^2(\rr,\rr^n)}^2}
\,. 
\end{equation}
A geometrical interpretation of this expression will be given below in a more specific case. 
\begin{remark} 
If $F(.)$ derives from a potential and $c=0$, then each one among expressions~\cref{solv_cond} and~\cref{solv_cond_alt} yields $\bar c=0$. Indeed, in this case, $DF(\phi)$ equals $DF(\phi)^*$ and $c$ equals $0$ and $\mathcal{L}$ equals $\mathcal{L}^*$, thus:
\begin{itemize}
\item it follows directly from \cref{solv_cond_alt} that $\bar c=0$;
\item or it follows from $\mathcal{L}=\mathcal{L}^*$ that $\psi$ and $\phi'$ are proportional, thus (since $\bar\ddd$ is assumed to be symmetric) \cref{solv_cond} yiels $\bar c=0$. 
\end{itemize}
\end{remark}
\subsection{Case of a two-dimensional reaction system with symmetries}
\label{subsec:with_symmetries}
Now let us consider a more specific situation, assuming that the reaction system is two-dimensional, and that the two ``species'' under consideration are completely symmetric for this system, before the perturbation. Thus, keeping the notation and assumptions of the previous \namecrefs{subsec:with_symmetries}, let us assume in addition that the dimension $n$ of the field variable $u$ equals two. Let us denote by $(u_1,u_2)$ the canonical coordinates of a vector $u$ in $\rr^2$, let $\sss$ denote the orthogonal symmetry exchanging the coordinates in $\rr^2$, namely 
\[
\sss:(u_1,u_2)\mapsto (u_2,u_1)
\,,
\]
and, from now on, let us make the following hypotheses:
\begin{description}
\item[(H3)] 
$F\circ\sss=\sss F
\quad\mbox{and}\quad
\ddd\sss=\sss\ddd
\quad\mbox{and, for all $x$ in $\rr$,}\quad
\phi(-x)=\sss\phi(x)
\,.$
\end{description}
In other words, we assume that both the reaction-diffusion system and the front $\phi(.)$ are $u_1\leftrightarrow u_2$-symmetric. 
\begin{lemma}[$c$ equals $0$]
\label{lem:c_equals_zero}
The speed $c$ equals $0$.
\end{lemma}
In other words, the front $\phi$ is a standing front.
\begin{proof}
For every real quantity $x$, system \cref{syst_front} considered at $-x$ reads
\[
-c\phi'(-x) = F\bigl(\phi(-x)\bigr) + \ddd \phi''(-x)
\] 
and this yields, according to (H3),
\[
c\sss \phi'(x) = \sss F\bigl(\phi(x)\bigr) + \sss \ddd \phi''(x)
\,,
\] 
and finally, getting rid of $\sss$ in this equality and comparing with system \cref{syst_front} considered at $x$, it follows that $c\phi'(x)$ equals $0$, and this proves \cref{lem:c_equals_zero}.
\end{proof}
Thus the operators $\mathcal{L}$ and $\mathcal{L}^*$ reduce to:
\[
\mathcal{L} = DF(\phi) + \ddd\partial_{\xi\xi}
\quad\mbox{and}\quad
\mathcal{L}^* = DF(\phi)^* + \ddd\partial_{\xi\xi}
\,.
\]
Since the matrix $\bar{\ddd}$ is not assumed to be $u_1\leftrightarrow u_2$-symmetric (in other words we do not assume that $\bar{\ddd}\sss=\sss\bar{\ddd}$), the perturbation  in \cref{r_diff_pert} in general breaks the $u_1\leftrightarrow u_2$-symmetry. 
For all $u$ in $\rr^2$, let us denote by $\rot F(u)$ the infinitesimal rotation of the vector field $F$. This quantity can be defined by:
\begin{equation}
\label{def_rot}
DF(u)-DF(u)^*=
\begin{pmatrix}
0&-\rot F(u)\\  \rot F(u)&0
\end{pmatrix}
\,.
\end{equation}
With this notation, expression~\cref{solv_cond_alt} becomes:
\begin{equation}
\label{bar_c_rot}
\bar c 
= \frac{\int_{-\infty}^{+\infty} \rot F\bigl(\phi(x)\bigr) \cdot\bigl(\phi'(x) \wedge \bar\varphi(x)\bigr)\, dx}{\int_{-\infty}^{+\infty}\phi'^2(x)\, dx}
= \frac{\bigl\langle \rot F(\phi),\phi' \wedge \bar\varphi\bigr\rangle_{L^2(\rr,\rr^n)}}{\norm{\phi'}_{L^2(\rr,\rr^n)}^2}
\,.
\end{equation}
\subsection{Geometric interpretation of the first order variation of the speed}
\label{subsec:geom_interpret_bar_c}
The last expression \cref{bar_c_rot} of $\bar{c}$ admits the following geometrical interpretation. Let us denote by $\Phi$ the image (the trajectory) in $\rr^2$ of the standing front $\phi$, that is: 
\[
\Phi = \{\phi(x): x\in\rr\}\subset \rr^2
\,.
\] 
The infinitesimal rotation $\rot F(\phi)$ measures the ``shear'' induced locally along $\Phi$ by the antisymmetric part of $DF$, and the real quantity $\phi'\wedge\bar\varphi$ is determined by the component of the perturbation $\bar\varphi$ that is orthogonal to $\Phi$ (see \cref{fig:barc_rot_gen}). The shear induced by $F$ acts on this transverse perturbation (it ``pushes'' towards $E_-$ or $E_+$, as illustrated on \cref{fig:barc_rot_gen}), and this results in a change for the speed that is given at first order in $\epsilon$ by the quantity $\bar c$ defined above.
\begin{figure}[!htbp]
	\centering
    \includegraphics[width=0.7\textwidth]{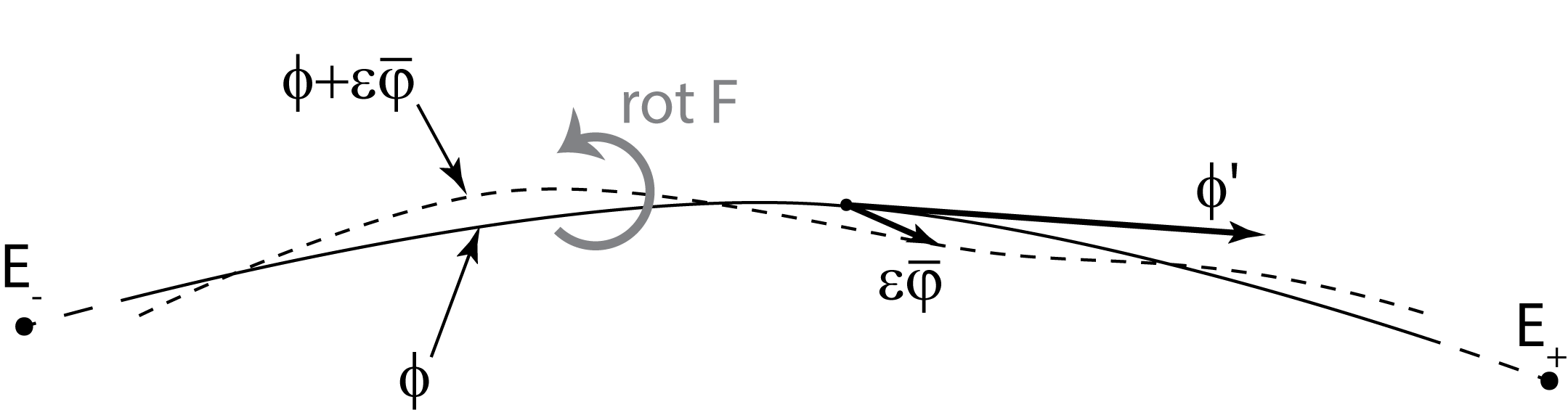}
    \caption{Geometrical illustration of expression~\cref{bar_c_rot} of $\bar c$.}
    \label{fig:barc_rot_gen}
\end{figure}
\subsection{Reduction using symmetry}
\label{subsec:red_sym}
The aim of this \namecref{subsec:red_sym} is to take into account the symmetries (H3) of the system to simplify expressions~\cref{solv_cond} and~\cref{bar_c_rot} of $\bar c$ (that is, to write the integrals in these expressions as integrals on $\rr_+$ only, instead of $\rr$). The first symmetries on the terms involved in these integrals are stated in the following lemma. 
\begin{lemma}[symmetries of $\rot F$ and $\psi$]
\label{lem:sym_psi_rotF}
For every real quantity $x$, 
\begin{equation}
\label{sym_psi_rotF}
\rot F\bigl(\phi(-x)\bigr) = - \rot F\bigl(\phi(x)\bigr)
\quad\mbox{and}\quad
\psi(-x)=-\sss\psi(x)
\,.
\end{equation}
\end{lemma}
\begin{proof}
It follows from the symmetry of $F$ with respect to $\sss$ in (H3) that, for every $u$ in $\rr^2$, 
\[
F(\sss u)=\sss F(u) 
\quad\mbox{thus}\quad
DF_{\sss u} \sss = \sss DF_{u}
\quad\mbox{and}\quad
DF_{\sss u} = \sss DF_{u} \sss
\,,
\]
and since $\sss^*$ equals $\sss$,
\[
DF_{\sss u}^* = \sss DF_{u}^* \sss
\,.
\]
It follows that
\[
DF_{\sss u} - DF_{\sss u}^* = \sss (DF_{u} - DF_{u}^*) \sss
\]
and according to the definition \cref{def_rot} of $\rot F(\cdot)$ it follows that
\[
\rot F(\sss u) = - \rot F(u)
\,.
\]
Thus, for every real quantity $x$, still according to (H3),
\[
\rot F \bigl(\phi(-x)\bigr) = \rot F \bigl( \sss \phi(x)\bigr) = - \rot F \bigl( \phi(x)\bigr)
\,.
\]
and this proves the first equality of \cref{sym_psi_rotF}.

To prove the second equality, let us consider the function $\eta$ defined by: $\eta(x)=\psi(-x)$. Then, according to the expression of $\mathcal{L}^*$ and to hypotheses (H3), 
\[
\begin{aligned}
\mathcal{L}^* (\sss \eta)(x) &= \ddd \sss \psi''(-x) + DF^*\bigl( \phi(x)\bigr) \sss \psi(-x) \\
&= \sss \ddd \psi''(-x) + DF^*\bigl( \sss \phi(-x)\bigr) \sss \psi(-x) \\
&= \sss \Bigl( \ddd \psi''(-x) + DF^*\bigl( \phi(-x)\bigr) \psi(-x) \Bigr) \\
&= \sss (\mathcal{L}^* \psi) (-x) \\
&=0
\,.
\end{aligned}
\]
In other words, the function $x\mapsto \sss\psi(-x)$ belongs to the eigenspace associated to the eigenvalue $0$ for the operator $\mathcal{L}^*$. Since this eigenspace is one-dimensional and contains the nonzero function $\psi$, it follows that there exists a real quantity $\lambda$ such that, for every real quantity $x$, 
\[
\sss\psi(-x)=\lambda \psi(x)
\,,
\]
and since the map
\[
L^2(\rr,\rr^2) \rightarrow L^2(\rr,\rr^2), \quad f \mapsto \bigl( x\mapsto \sss f(-x) \bigr)
\]
is an involution, it follows that $\lambda=\pm 1$. According to the symmetry of $\phi$ with respect to $\sss$ in (H3), for every real quantity $x$, 
\[
\phi'(-x)=-\sss\phi'(x)
\quad\mbox{and}\quad
\psi(-x) = \lambda\sss\psi(x)
\,.
\]
This shows $\lambda$ cannot be equal to $1$, or else the function $x\mapsto \psi(x)\cdot\phi'(x)$ would be odd, and the scalar product $\langle\psi,\phi'\rangle_{L^2(\rr,\rr^n)}$ would vanish, whereas according to the assumptions we made this scalar product must be nonzero (and was actually normalized to $1$). Thus $\lambda$ equals $-1$, and this proves the second equality of \cref{sym_psi_rotF}. \Cref{lem:sym_psi_rotF} is proved.
\end{proof}
Since we are interested in the effect of breaking the $u_1\leftrightarrow u_2$-symmetry of the diffusion matrix, it is convenient to assume that the $u_1\leftrightarrow u_2$-symmetric part of the (symmetric) matrix $\bar{\ddd}$ vanishes (and that the $u_1\leftrightarrow u_2$-antisymmetric part of the same matrix does not vanish). This is exactly the meaning of our next hypothesis:
\begin{description}
\item[(H4)] $\sss\bar{\ddd}=-\bar{\ddd}\sss$ and $\bar{\ddd}\not=0$. 
\end{description}
According to this hypothesis, there exists a nonzero real quantity $d$ such that:
\[
\bar{\ddd}=
\begin{pmatrix}
d&0\\ 0&-d
\end{pmatrix}
\,.
\]
This hypothesis leads to the following additional symmetry. 
\begin{lemma}[symmetry of $\bar\varphi$]
\label{lem:sym_of_bar_phi}
For every real quantity $x$, 
\[
\bar\varphi(-x)=-\sss\bar\varphi(x)
\,.
\]
\end{lemma}
\begin{proof}
Since $\bar\varphi$ is a solution of system \cref{1st_ord_front}, for every real quantity $x$, 
\[
\ddd \bar\varphi'' (x) + DF\bigl(\phi(x)\bigr) \bar\varphi (x) = - \bar{\ddd} \phi''(x) - \bar c \phi'(x)
\,.
\]
Multiplying (to the left) by $\sss$ both sides of this equality and using the symmetries (H3) and (H4), it follows that
\[
\ddd \sss \bar\varphi'' (x) + DF\bigl(\phi(-x)\bigr) \sss\bar\varphi (x) = \bar{\ddd} \phi''(-x) + \bar c \phi'(-x)
\,,
\]
and this shows that the function $x\mapsto -\sss \bar\varphi(-x)$ is also a solution of system \cref{1st_ord_front}. Observe in addition that according to the symmetry of $\phi$ this solution is orthogonal to $\phi'$ for the $L^2(\rr,\rr^n)$-scalar product. Thus this solution must be equal to $\bar\varphi$, and this proves the lemma. 
\end{proof}
\begin{lemma}[even integrands]
\label{lem:even_integrands}
The three functions  
\[
x\mapsto \psi(x)\cdot\bar{\ddd}\phi''(x)
\quad\mbox{and}\quad
x\mapsto \phi'(x)^2
\quad\mbox{and}\quad
x\mapsto \rot F\bigl(\phi(x)\bigr)\cdot\bigl(\phi'(x)\wedge\bar\varphi(x)\bigl)
\]
are even.
\end{lemma}
\begin{proof}
For the two first functions, the symmetry follows from (H3) and \cref{lem:sym_psi_rotF}. For the third function, observe that, for every real quantity $x$, according to (H3) and \cref{lem:sym_of_bar_phi},
\[
\phi'(-x)\wedge\bar\varphi(-x) = \sss \phi'(x)\wedge \sss\bar\varphi(x) = - \phi'(x)\wedge\bar\varphi(x)
\,.
\]
and the result follows from \cref{lem:sym_psi_rotF}.
\end{proof}
It follows from \cref{lem:even_integrands} that expressions~\cref{solv_cond} and~\cref{bar_c_rot} of $\bar c$ can be rewritten with integrals restricted to $\rr_+$, namely:
\begin{align}
\label{s_cond_rr_plus}
\bar c 
&= -2 \int_{0}^{+\infty} \psi(x)\cdot\bar{\ddd}\phi''(x)\,dx
= -2 \langle\psi,\bar{\ddd}\phi''\rangle_{L^2(\rr_+,\rr^n)}
\\
\label{c_rot_rr_plus}
&= \frac{\int_{0}^{+\infty} \rot F\bigl(\phi(x)\bigr)\cdot\bigl(\phi'(x) \wedge \bar\varphi(x)\bigr)\,dx}{\int_{0}^{+\infty}\phi'^2(x)\,dx}
= \frac{\langle \rot F(\phi),\phi' \wedge \bar\varphi\rangle_{L^2(\rr_+,\rr^n)}}{\norm{\phi'}_{L^2(\rr_+,\rr^n)}^2}
\,.
\end{align}
The aim of the two following \namecrefs{sec:toy} is to compute the sign of the quantity $\bar c$ on two specific examples. 
\section{Toy example}
\label{sec:toy}
\subsection{Definition}
\label{subsec:toy_def}
The aim of this \namecref{sec:toy} is to show on a toy example that both signs can occur for the quantity $\bar c$. 
Let $u=(u_1,u_2)$ denote again the canonical coordinates in $\rr^2$, let $\mu$ denote a real quantity (a parameter), and let us consider the following system (see \cref{fig:phase_toy}): 
\begin{figure}[!htbp]
	\centering
    \includegraphics[width=0.6\textwidth]{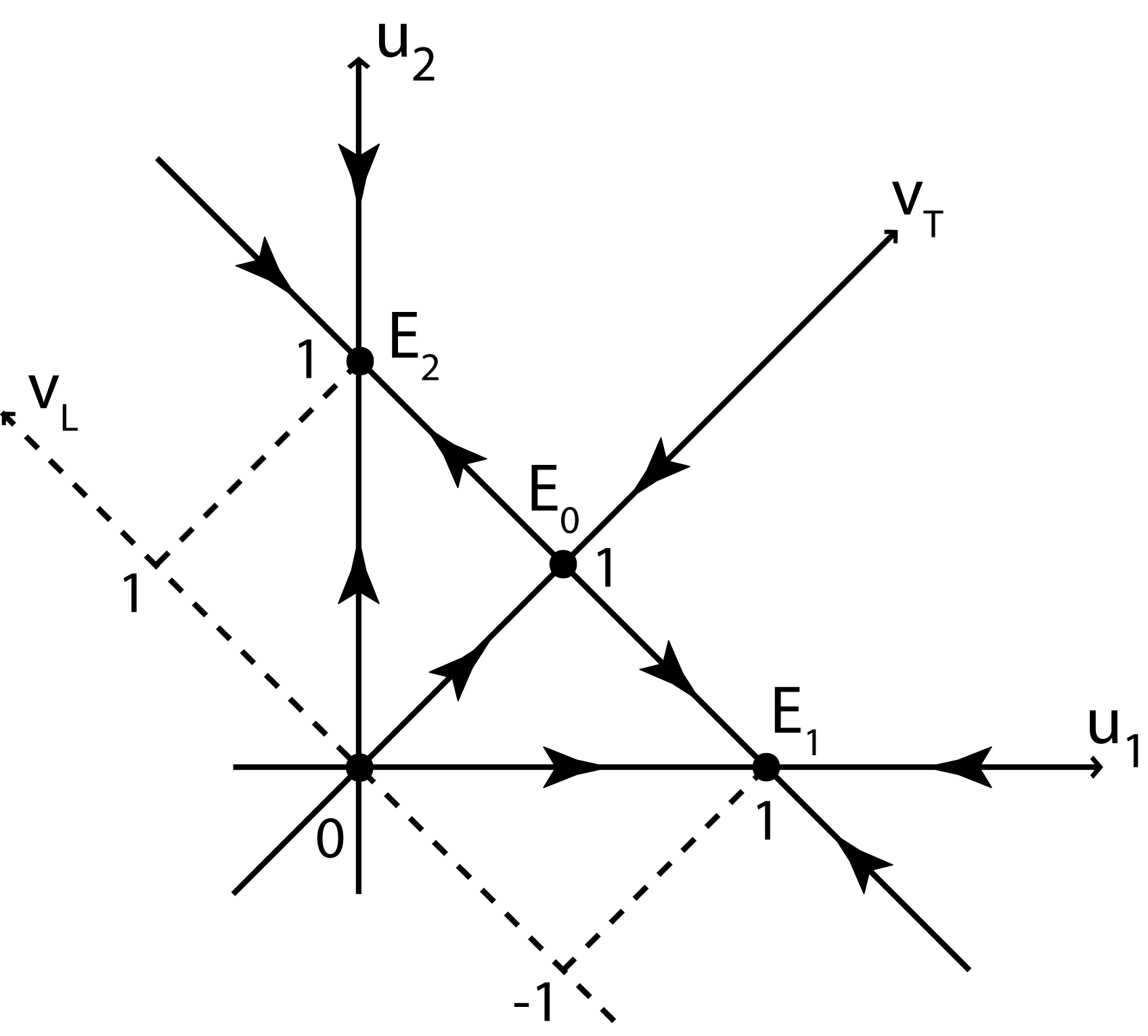}
    \caption{Phase space of the reaction system $u_t=F_\mu(u)$ for $\mu>0$.}
    \label{fig:phase_toy}
\end{figure}
\begin{equation}
\label{syst_toy_u}
u_t = F_\mu(u) + u_{xx}
\quad\mbox{with}\quad
F_\mu(u) = 
F_\mu\begin{pmatrix}u_1 \\ u_2\end{pmatrix}=
\begin{pmatrix}
u_1\bigl(1-(u_1+u_2)-\mu u_2(u_2-u_1)\bigr)\\  u_2\bigl(1-(u_1+u_2)-\mu u_1(u_1-u_2)\bigr)
\end{pmatrix}
. 
\end{equation}
Both axes $\{u_2=0\}$ and $\{u_1=0\}$ are invariant under the reaction system $u_t=F_\mu(u)$, and the restriction of this system to each of these axes is nothing but the logistic equation: $w_t=w(1-w)$. 
Besides, the $u_1\leftrightarrow u_2$-symmetry $F_\mu\circ\sss=\sss F_\mu$ clearly holds.
\begin{notation}
Let us consider the following alternative coordinate system $v=(v_T,v_L)$ related to $u=(u_1,u_2)$ by:
\begin{equation}
\label{def_trans_long_coord}
\left\{
\begin{aligned}
v_T &= u_1+u_2 \\
v_L &= -u_1+u_2
\end{aligned}
\right.
\Longleftrightarrow
\left\{
\begin{aligned}
u_1 &= \frac{v_T-v_L}{2} \\
u_2 &= \frac{v_T+v_L}{2}
\end{aligned}
\right.
\end{equation}
(see \cref{fig:phase_toy}). The subscripts ``$T$'' and ``$L$'' refer to the adjectives ``transversal'' and ``longitudinal'', with respect to the standing front $\phi$ that will be defined below. Along this \namecref{sec:toy} and the next one, these subscripts will always be used to denote the coordinates of a point in this ``transversal-longitudinal'' coordinate system, whereas the subscripts ``$1$'' and ``$2$'' will always be used to denote the canonical coordinates. 
\end{notation}
When expressed within the transversal-longitudinal coordinate system, system \cref{syst_toy_u} takes the form
\begin{equation}
\label{syst_toy_v}
v_t=G_\mu(v)+v_{xx}
\quad\mbox{with}\quad
G_\mu(v)=G_\mu(v_T,v_L)=
\begin{pmatrix}
v_T-v_T^2 \\
v_L\Bigl(1-v_T+\frac{\mu}{2}(v_T^2-v_L^2)\Bigr)
\end{pmatrix}
\,.
\end{equation}
According to this expression, the line $\{v_T=1\}$ is invariant (and transversely attractive), and the restriction of system \cref{syst_toy_v} to this line reads:
\[
\partial_t v_L=\frac{\mu}{2}v_L(1-v_L^2)+\partial_{xx}v_L
\,.
\] 
Thus, if the parameter $\mu$ is negative, the reaction system is monostable, with a unique stable equilibrium $E_0$ at $(v_T,v_L)=(1,0)$, whereas if $\mu$ is positive then it is bistable, with two stable equilibria $E_1$ at $(v_T,v_L)=(1,-1)$ and $E_2$ at $(v_T,v_L)=(1,1)$, and a saddle $E_0$ at $(v_T,v_L)=(1,0)$ (see \cref{fig:phase_toy}). 
By the way,
\begin{equation}
\label{lin_asymt_toy}
DF_\mu(E_1) = 
\begin{pmatrix}
-1 & 1+\mu \\ 0 & -\mu
\end{pmatrix}
,\quad
DF_\mu(E_2) = 
\begin{pmatrix}
-\mu & 0 \\ 1+\mu & -1
\end{pmatrix}
,\quad
DG_\mu(E_0) = 
\begin{pmatrix}
-1 & 0 \\ 0 & \frac{\mu}{2}
\end{pmatrix}
\end{equation}
(compare with expression \vref{lin_asymot_lv} for the Lotka--Volterra competition system).
\subsection{Standing front}
\label{subsec:toy_standing_front}
Let us assume from now on that $\mu$ is positive (bistable case). In this case there exists for systems \cref{syst_toy_u,syst_toy_v} a standing front $x\mapsto \phi(x)$ connecting $E_1$ to $E_2$, which is given (in transversal-longitudinal coordinates) by the explicit formula: 
\begin{equation}
\label{toy_front}
\phi_T(x)\equiv 1
\quad\mbox{and}\quad
\phi_L(x)=\tanh\frac{\sqrt{\mu}x}{2}
\end{equation}
(the connection with the notation used in \cref{sec:general} is obvious: equilibria $E_1$ and $E_2$ defined above correspond to equilibria $E_-$ and $E_+$ of \cref{sec:general}, respectively). 

This standing front satisfies the $u_1\leftrightarrow u_2$-symmetry, that is $\phi(-x)=\sss\phi(x)$ for every real quantity $x$. All the symmetry hypotheses (H3) are therefore satisfied for the system \cref{syst_toy_u} and the standing front $\phi$. 

For the remaining of \cref{sec:toy} we shall mainly work with the transversal-longitudinal coordinate system.
The linear operator $\mathcal{L}$ (obtained by linearising system \cref{syst_toy_v} around this standing front) reads, expressed in these coordinates, 
\begin{equation}
\label{expression_L_toy}
\mathcal{L}
\begin{pmatrix}
\varphi_T \\ 
\varphi_L
\end{pmatrix}
=
\begin{pmatrix}
-1 & 0 \\ 
(\mu-1)\phi_L & \frac{\mu}{2}(1-3\phi_L^2)
\end{pmatrix}
\begin{pmatrix}
\varphi_T \\ 
\varphi_L
\end{pmatrix}
+
\begin{pmatrix}
\varphi_T'' \\ 
\varphi_L''
\end{pmatrix}
\,.
\end{equation}
\begin{lemma}[spectral stability of the standing front, toy example]
\label{lem:lin_st_toy}
The standing front \\
$x\mapsto \phi(x)$ is spectrally stable.
\end{lemma}
\begin{proof}
According to expressions \cref{lin_asymt_toy}, the essential spectrum of $\mathcal{L}$ is the interval:
\[
(-\infty,\max(-1,-\mu)] 
\quad\mbox{included in}\quad
(-\infty,0)
\,.
\]
A function 
\[
x\mapsto \varphi(x)=\bigl(\varphi_T(x),\varphi_L(x)\bigr)
\]
is an eigenfunction of $\mathcal{L}$ for an eigenvalue $\lambda$ if and only if $\varphi_T$ vanishes identically and $\varphi_L$ is an eigenfunction of the operator
\begin{equation}
\label{lin_toy_v2}
\ell_{\mu}:\varphi_L \mapsto \frac{\mu}{2}(1-3\phi_L^2) \varphi_L + \varphi_L''
\end{equation}
for the same eigenvalue $\lambda$. 
The function $x\mapsto\varphi'_L(x)$ is an eigenfunction of $\ell_{\mu}$ for the eigenvalue zero (this comes from translation invariance in space), and by a standard Sturm--Liouville argument (\cite{CoddingtonLevinson_theoryODE_1955}), this eigenvalue is simple and all other eigenvalues (they are real since $\ell_{\mu}$ is a self-adjoint operator for the $L^2(\rr,\rr)$-scalar product) are negative. \Cref{lem:lin_st_toy} is proved. 
\end{proof}
\subsection{First order variation of the front speed}
\label{subsec:toy_front_speed}
According to the notation of \cref{subsec:setup}, $\ddd=\id_{\rr^2}$. Let us choose the perturbation matrix $\bar{\ddd}$ as follows: 
\begin{equation}
\label{def_bar_D_toy}
\begin{aligned}
\bar{\ddd} & =
\begin{pmatrix}
1&0\\  0&-1
\end{pmatrix}
\mbox{ in canonical coordinates,}\\
\mbox{or equivalently }
\bar{\ddd} & =
\begin{pmatrix}
0&-1\\  -1&0
\end{pmatrix}
\mbox{ in transversal-longitudinal coordinates.}
\end{aligned}
\end{equation}
All hypotheses (H1-4) of \cref{sec:general} are satisfied. Let us keep the notation $\mathcal{L}^*$ and $\psi$ and $\bar c$ and $\bar \varphi$ introduced there. The following result shows that both signs for $\bar c$ may occur, depending on the value of $\mu$. 
\begin{proposition}[sign of the first order variation of front speed, toy example]
\label{prop:bar_c_toy}
The \\
sign of $\bar c$ equals that of $1-\mu$; that is,
\[
\begin{aligned}
\bar 0<\bar c \quad & \mbox{if}\quad 0<\mu<1 \,, \\
\mbox{and}\quad \bar c=0 \quad & \mbox{if}\quad \mu=1 \,, \\
\mbox{and}\quad \bar c<0 \quad & \mbox{if}\quad 1<\mu 
\,.
\end{aligned}
\]
\end{proposition}
This proposition can be understood as follows.
\begin{itemize}
\item For $\mu$ in $(0,1)$, the perturbation promotes $E_1$. And since the perturbation increases the mobility of the species corresponding to $E_1$, this shows that an increase of mobility is advantageous in this case.
\item For $\mu$ larger than $1$, the perturbation promotes $E_1$. And for the same reason, this time, an increase of mobility turns out to be disadvantageous. 
\end{itemize}
Before proving this proposition, let us begin with a geometrical interpretation (this interpretation will by the way provide an informal proof, and make the proof easier to follow). 
\begin{figure}[!htbp]
	\centering
    \includegraphics[width=\textwidth]{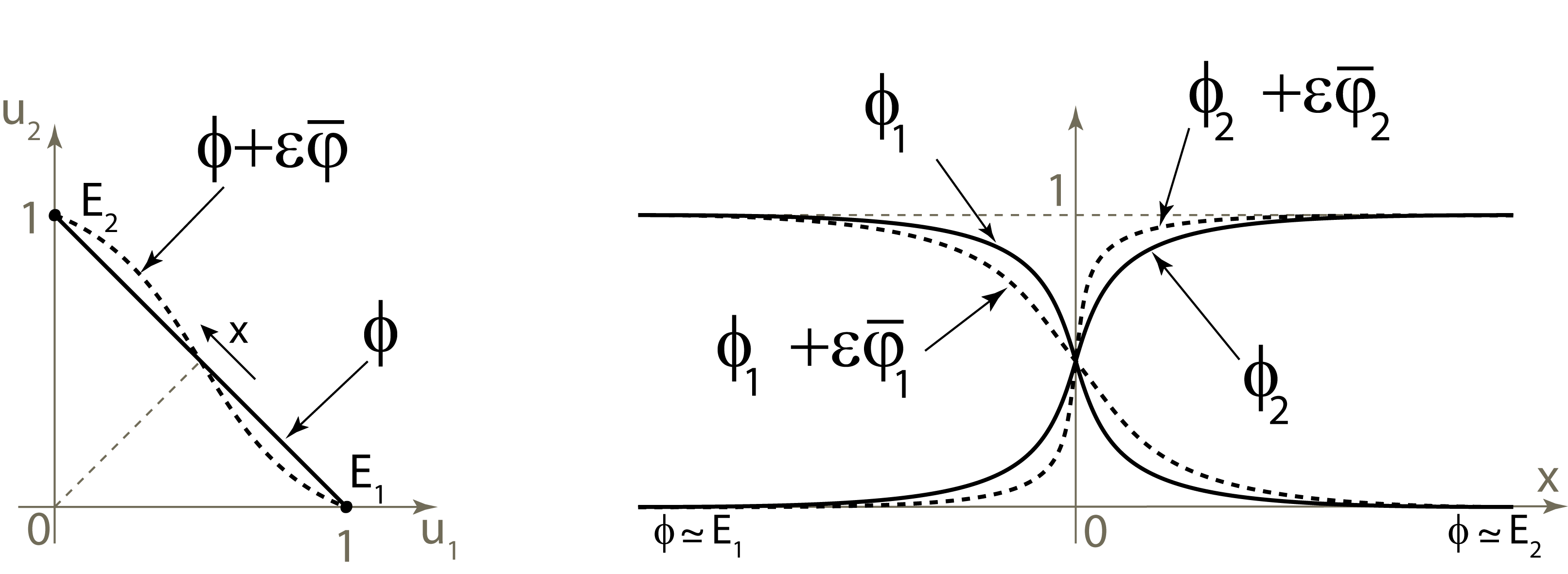}
    \caption{Left: the initial standing front $\phi$ and its perturbation $\phi+\epsilon\bar\varphi$. Right: the graphs of the components of these two fronts in the canonical coordinate system. The perturbation consists in an increase of the mobility of the first species (equilibrium $E_1$) and a decrease of the mobility of the second species (equilibrium $E_2$).}
    \label{fig:front_toy}
\end{figure}
\begin{figure}[!htbp]
	\centering
    \includegraphics[width=0.9\textwidth]{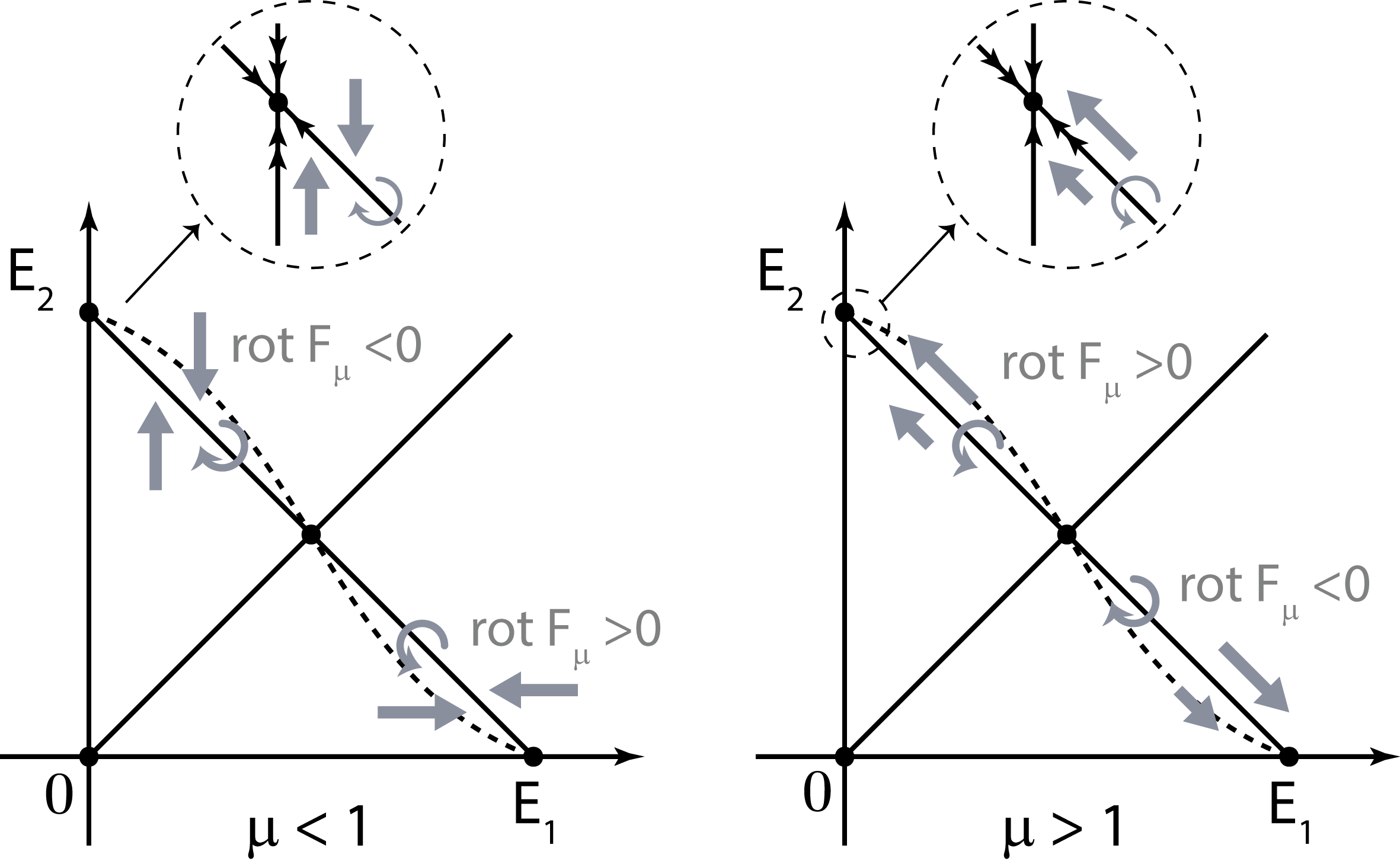}
    \caption{Orientation of the shear along the standing front in the ``moderately strong'' ($\mu$ smaller than $1$) versus ``hard'' ($\mu$ larger than $1$) competition regimes. In the moderately strong competition regime, being more mobile is an advantage, and the most mobile species (the first one, equilibrium $E_1$) wins with respect to the second species (equilibrium $E_2$). In the the hard competition regime, being more mobile is a disadvantage, and the most mobile species (the first one) looses.}
    \label{fig:shear_toy}
\end{figure}
Let $(\phi_1,\phi_2)$ denote the components of $\phi$ in the canonical coordinate system, and let $(\bar\varphi_1,\bar\varphi_2)$ denote the components of $\bar\varphi$ in the same canonical coordinate system. As illustrated on \cref{fig:front_toy}, the functions $\phi_1$ and $\phi_2$ are symmetric. The perturbation breaks this symmetry: the mobility of the first species has been slightly increased, while the mobility of the second species has been slightly decreased. As a consequence, one expects that the graph of $\phi_1+\epsilon\bar\varphi_1$ will be slightly flatter than that of $\phi_1$, and conversely that the graph of $\phi_2+\epsilon\bar\varphi_2$ will be slightly straighter than that of $\phi_2$ (\cref{fig:front_toy}). As a consequence, the position of the image of the perturbed front with respect to the image of the initial standing front should be as illustrated on \cref{fig:front_toy}; it suggests that $\bar\varphi_T(x)$ --- the first component of $\bar\varphi$ in the transversal-longitudinal coordinate system --- should be negative for $x$ negative and positive for $x$ positive (the proof below will confirm this). 

On the other hand, expression \cref{syst_toy_v} of the reaction-diffusion system in the transversal-longitudinal coordinate system yields, or all $x$ in $\rr$:
\begin{equation}
\label{rot_toy}
\rot F_\mu\bigl(\phi(x)\bigr)=(\mu-1)\phi_L(x)
\,. 
\end{equation}
Thus the sign of the shear induced by $F_\mu$ along $\phi$ depends on the sign of $\mu-1$ (see \cref{fig:shear_toy}). In view of \cref{fig:front_toy,fig:shear_toy}, it could be expected that, for $\mu<1$, the perturbation is in favour of $E_1$, while for $\mu>1$ it is in favour of $E_2$, as stated by \cref{prop:bar_c_toy}. 

Here is another possible interpretation. The parameter $\mu$ represents a sort of ``intensity'' of the competition between the two species. If $\mu$ is positive but smaller than $1$, the intensity can be qualified as ``moderately strong''. In this case, as illustrated on \cref{fig:shear_toy} (see the zoom on equilibrium $E_2$), the dominant effect of the reaction term is to balance the total density $u_1+u_2$ (to drive this total density to $1$), and this turns out to be in favour of the most mobile species. On the other hand, if $\mu$ is larger than $1$, then the intensity of the competition can be qualified as ``hard''.  It this case, the dominant effect of the reaction term is to drive the system in favour of the most represented species locally (and away of the $E_0$ saddle equilibrium where both densities are equal). And this turns out to be in favour of the less mobile species. In short (you may apply this to you everyday life \Smiley): if the struggle is moderate, spread away to gain new territories; if it is bloody, avoid the dispersal and concentrate your forces~! 

Let us now prove \cref{prop:bar_c_toy}.
Let us denote by $(\psi_T,\psi_L)$ and by $(\bar\varphi_T,\bar\varphi_L)$ the transversal-longitudinal coordinates of the functions $\psi$ and $\bar\varphi$. To prove \cref{prop:bar_c_toy}, each one among expressions \vref{s_cond_rr_plus,c_rot_rr_plus} can be used (resulting in two different proofs). The two proofs are given below, beginning with the proof involving expression~\cref{c_rot_rr_plus}, since it is closer to the geometrical interpretation above. 
\begin{proof}[Proof using expression~\cref{c_rot_rr_plus}] 
According to expression \cref{c_rot_rr_plus}, the sign of $\bar c$ is equal to the sign of:
\begin{equation}
\label{sign_bar_c_toy}
\int_{0}^{+\infty} \rot F_\mu\bigl(\phi(x)\bigr)\cdot\bigl(\phi'(x) \wedge \bar\varphi(x)\bigr)\,dx
\,.
\end{equation}
According to expression \cref{rot_toy} of $\rot F_\mu\bigl(\phi(x)\bigr)$ and expression \cref{toy_front} of $\phi_L(x)$ the function $\rot F_\mu\bigl(\phi(\cdot)\bigr)$ is of the sign of $\mu-1$ on $\rr_+$. 

On the other hand, since $\phi_T'$ is identically zero, the function $\phi'\wedge\bar\varphi$ equals $-\phi'_L\bar\varphi_T$, and according to expression \cref{toy_front} of $\phi_L(x)$ the function $\phi_L'$ is positive on $\rr_+$. It remains to determine the sign of $\bar\varphi_T$ on $\rr_+$. Projecting system \cref{1st_ord_front} on the $v_T$-axis yields: 
\begin{equation}
\label{proj_toy}
\bar\varphi_T''=\phi_L''+\bar\varphi_T
\,. 
\end{equation}
Recall that according to \vref{lem:sym_of_bar_phi}, the quantities $\bar\varphi(-x)$ and $-\sss\bar\varphi(x)$ are equal for every $x$ in $\rr$; as a consequence, since the transversal coordinate is unchanged by $\sss$, the quantity $\bar\varphi_T(0)$ must vanish. Since $\bar\varphi_T(x)$ approaches $0$ when $x$ approaches $+\infty$ and since according to \cref{toy_front} $\phi''_L(x)$ is negative for all $x$ in $\rr_+^*$, it follows from equation~\cref{proj_toy} shows that the function $\bar\varphi_T$ is positive on $\rr_+^*$ (see \vref{lem:sol_forced_2nd_order} in appendix). It follows that the function $\phi'\wedge\bar\varphi$ is negative on $\rr_+^*$, and this proves \cref{prop:bar_c_toy}. 
\end{proof}
\begin{proof}[Proof using expression~\cref{s_cond_rr_plus}]
According to expression \vref{s_cond_rr_plus}, we have:
\[
\bar c 
=- 2\int_{0}^{+\infty}\psi(x)\cdot\bar{\ddd}\phi''(x)\, dx 
= 2\int_{0}^{+\infty}\psi_T(x)\cdot\phi''_L(x)\, dx
\,,
\]
and we know from the explicit expression \cref{toy_front} of $\phi_L$ that the quantity $\phi''_L(x)$ is negative for all $x$ in $\rr_+^*$. Therefore all we have to do is show that $\psi_T$ and $\mu-1$ have the same sign (and vanish at the same time).  

According to the expression \vref{expression_L_toy} of $\mathcal{L}$, system $\mathcal{L}^*\psi=0$ reads (using the notation $\ell_{\mu}$ introduced in definition \cref{lin_toy_v2}):
\begin{align}
-\psi_T + (\mu-1)\phi_L\psi_L + \psi_T'' &= 0 \label{psi_T_toy}
\,, \\
\ell_{\mu}\psi_L &= 0 \label{psi_L_toy}
\,.
\end{align}
Since the eigenvalue zero of $\ell_{\mu}$ is simple (see the proof of \cref{lem:lin_st_toy} above), equation \cref{psi_L_toy} shows that the functions $\psi_L$ and $\phi'_L$ must be proportional. Thus $\psi_L = N\phi'_L$, where, according to the normalizing condition \vref{norm_cond}, the normalizing constant $N$ is:
\[
N = \Bigl(\int_{-\infty}^{+\infty}\phi_L'^2(x)\,dx\Bigr)^{-1} 
= \norm{\phi_L'}_{L^2(\rr,\rr)}^{-2}>0
\,.
\]
Thus, according to equation \cref{psi_T_toy}, the following differential equation holds for $\psi_T$:
\begin{equation}
\label{adjoint_toy}
\psi_T'' = \psi_T+N(1-\mu)\phi_L\phi'_L
\,.
\end{equation}
Recall that according to \vref{lem:sym_psi_rotF}, the quantities $\psi(-x)$ and $-\sss\psi(x)$ are equal for every $x$ in $\rr$; as a consequence, since the transversal coordinate is unchanged by $\sss$, the quantity $\psi_T(0)=0$ must vanish. Since $\psi_T(x)$ approaches $0$ when $x$ approaches $+\infty$, and since both quantities $\phi_L(x)$ and $\phi'_L(x)$ are positive for all $x$ in $\rr_+^*$, this shows that the sign of $\psi_T$ must remain constant and opposite to that of $1-\mu$ on $\rr_+^*$  (see \vref{lem:sol_forced_2nd_order} in appendix), and that $\psi_T$ vanishes identically if $\mu$ is equal to $1$. \Cref{prop:bar_c_toy} is proved. 
\end{proof}
\section{Bistable Lotka--Volterra competition model}
\label{sec:lv}
\subsection{Definition}
\label{subsec:lv_def}
\begin{figure}[!htbp]
	\centering
    \includegraphics[width=0.6\textwidth]{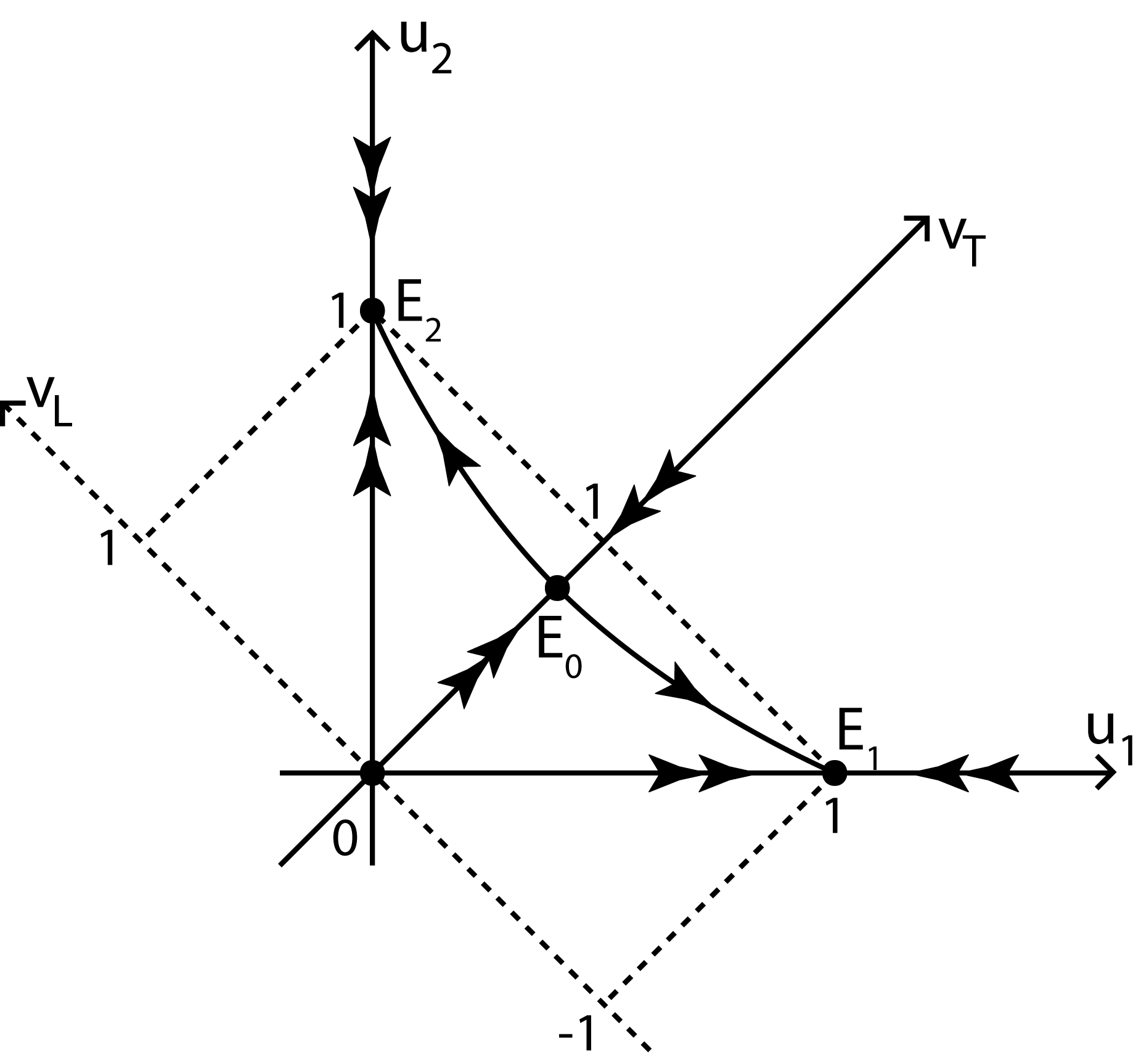}
    \caption{Phase space of the reaction system in the bistable case ($\mu>0$).}
    \label{fig:phase_lv}
\end{figure}
Let $\mu$ denote a real quantity (a parameter) and let us consider the following reaction-diffusion system, where the reaction term is known as the Lotka--Volterra competition model (see \cref{fig:phase_lv}): 
\begin{equation}
\label{syst_lv_u}
u_t = F_\mu(u) + u_{xx}
\quad\mbox{with}\quad
F_\mu(u)=F_\mu(u_1,u_2)=
\begin{pmatrix}
u_1\bigl(1-u_1-(1+\mu) u_2\bigr)\\  u_2\bigl(1-u_2-(1+\mu) u_1\bigr)
\end{pmatrix}
\,.
\end{equation}
Both axes $\{u_2=0\}$ and $\{u_1=0\}$ are invariant under the reaction differential system $u_t=F_\mu(u)$, and the restriction of this system to each of these axes is nothing but the logistic equation $w_t = w(1-w)$. 
The $u_1\leftrightarrow u_2$-symmetry $F_\mu\circ\sss=\sss F_\mu$ holds.  

Again in this \namecref{sec:lv}, we are going to use the ``transversal-longitudinal'' coordinate system $v=(v_T,v_L)$ defined exactly as in definition \vref{def_trans_long_coord}. 
Expressed in this coordinate system, the reaction-diffusion system \cref{syst_lv_u} takes the  form:
\begin{equation}
\label{syst_lv_v}
v_t=G_\mu(v)+v_{xx}
\quad\mbox{with}\quad
G_\mu(v)=G_\mu(v_T,v_L)=
\begin{pmatrix}
v_T-v_T^2+\frac{\mu}{2}(v_L^2-v_T^2)\\  v_L(1-v_T)
\end{pmatrix}
. 
\end{equation}
For all $\mu$ in $\rr\setminus\{-2,0\}$, the reaction system admits three equilibria aside from $(0,0)$ (see \cref{fig:phase_lv}):
\begin{itemize}
\item $E_1$, that is: $(u_1,u_2)=(1,0)\Leftrightarrow (v_T,v_L)=(-1,1)$,
\item $E_2$, that is: $(u_1,u_2)=(0,1)\Leftrightarrow (v_T,v_L)=(1,1)$,
\item $E_0$, that is: $(u_1,u_2)=\bigl(1/(2+\mu),1/(2+\mu)\bigr)\Leftrightarrow (v_T,v_L)=\bigl(1/(1+\mu/2),0\bigr)$.
\end{itemize}
The differential of the reaction system reads:
\[
DF_\mu(u_1,u_2)=
\begin{pmatrix}
1-2u_1-(1+\mu) u_2 & -(1+\mu)u_1 \\  -(1+\mu)u_2 & 1-2u_2-(1+\mu) u_1
\end{pmatrix}
\]
and
\[
DG_\mu(v_T,v_L)=
\begin{pmatrix}
1-(2+\mu)v_T & \mu v_L\\  -v_L & 1-v_T
\end{pmatrix}
\]
thus
\begin{equation}
\label{lin_asymot_lv}
DF_\mu(E_1) = 
\begin{pmatrix}
-1 & -1-\mu \\ 0 & -\mu
\end{pmatrix}
,\ 
DF_\mu(E_2) = 
\begin{pmatrix}
-\mu & 0 \\ -1-\mu & -1
\end{pmatrix}
,\ 
DG_\mu(E_0) = 
\begin{pmatrix}
-1 & 0 \\ 0 & \frac{\mu}{2+\mu}
\end{pmatrix}
.
\end{equation}
For $\mu$ negative, the reaction system is monostable ($E_0$ is stable and $E_2$ and $E_1$ are saddles) while for $\mu$ positive it is bistable ($E_2$ and $E_1$ are stable and $E_0$ is a saddle). From now on, it will be assumed that:
\[
\mu >0
\,,
\]
or in other words that the interspecific competition rate $1+\mu$ is higher than the intraspecific competition rate $1$, or in other words that the reaction system is bistable. 
Observe that (with the notation of \cref{sec:general}), the infinitesimal rotation of the vector field $G_\mu$ reads: 
\[
\rot G_\mu(v_T,v_L) = -(1+\mu)v_L
\,.
\]
Thus, by contrast with the toy example studied in the previous \namecref{sec:toy}, the sign of the shear along the trajectory of the expected front connecting $E_1$ to $E_2$ does not depend on the parameter $\mu$ (since $\mu$ is assumed to be positive). Thus in view of the observations made in the previous \namecref{sec:toy} we may expect that an increase of mobility is in this case always advantageous, in other words that the quantity $\bar c$ is positive. The aim of this \namecref{sec:lv} is to prove this statement when $\mu$ is altogether positive and small. 
\subsection{Standing front}
\label{subsec:lv_standing_front}
A smooth function 
\[
\phi:\rr\rightarrow\rr^2,\quad x\mapsto\phi(x) = \bigl(\phi_T(x),\phi_L(x)\bigr)
\]
is a stationary solution of system \cref{syst_lv_v} if it is a solution of
\begin{equation}
\label{syst_front_lv}
G_\mu(\phi)+\phi''=0
\Longleftrightarrow 
\left\{
\begin{aligned}
\phi_T'' &= \phi_T^2-\phi_T+\frac{\mu}{2}(\phi_T^2-\phi_L^2) \\ 
\phi_L'' &= \phi_L(\phi_T-1).
\end{aligned}
\right.
\end{equation}
Standing or travelling bistable fronts for systems including \cref{syst_lv_u} have been studied by many authors for a long time. Existence and asymptotic stability of a bistable (monotone) travelling front connecting $E_1$ to $E_2$ were first established (in a more general setting) by C. Conley and R. Gardner using topological methods and comparison principles, \cite{Gardner_exisStabVWDegreeTheoreticApproach_1982,ConleyGardner_applicationGeneralizedMorseIndexTW_1984}. 

In \cite{Kan-on_parameterDependencePropSpeedTW_1995,Kan-onFang_stabilityMonotoneTW_1996}, Y. Kan-on and Q. Fang proved the uniqueness of this bistable travelling front and its spectral stability (including its transversality/robustness that is the fact that the eigenvalue $0$ is simple) for Lotka--Volterra competition-diffusion systems (a class of systems including \cref{syst_lv_u} and governed by three reduced parameters aside of diffusion coefficients); from these spectral properties they recovered the asymptotic stability of this bistable front. In \cite{Kan-on_parameterDependencePropSpeedTW_1995}, Kan-on also proved the monotonicity of the speed of the bistable front with respect to the parameters characterizing the reaction system (but not with respect to the diffusion coefficients of the two species). Further insight into the sign of the speed of the front were achieved by J-S Guo and Y-C Lin, \cite{GuoLin_signWaveSpeed_2013}. However, here again their results are mainly concerned with the dependence of this sign with respect to the parameters of the reaction system, but not with respect to the diffusion coefficients of the two components, and little is stated about the (rather specific) case considered in this paper, where the reaction system is $u_1\leftrightarrow u_2$-symmetric, and where the sole breaking of this $u_1\leftrightarrow u_2$-symmetry comes from the diffusion coefficients. 

In \cite{GirardinNadin_tWRelativeMotilityInvasionSpeed_2015}, L. Girardin and G. Nadin also studied the dependence of the front speed with respect to the coefficients of the system, and this time espescially with respect to the diffusion coefficients of the two species. Their results hold when the parameters of the system approach certain limits; for the more restricted system \cref{syst_lv_u}, this corresponds to the limit when $\mu$ approaches $+\infty$ (in other words, when the interspecific competition rates approach $+\infty$). Their main ``Unity is not strength'' theorem states that, close to this limit, it is the most mobile species that dominates the other one. Surprisingly enough, this fits with the ``moderately strong'' competition case of the previous toy example, but not with the ``hard'' competition case where by contrast it was the less motile species that was dominant. 

Our purpose is to consider system \cref{syst_lv_u} when the parameter $\mu$ is positive and small (thus an asymptotics completely different from the one considered by Girardin and Nadin). By contrast with the previous toy example, the standing front connecting $E_1$ to $E_2$ is (to the knowledge of the author) not given by an explicit expressions. Note that explicit expressions for standing or travelling waves of Lotka--Volterra competition-diffusion systems have been provided by various authors, for instance by M. Rodrigo and M. Mimura in \cite{MimuraRodrigo_exactSolutionsCompetitionDiffusion_2000,MimuraRodrigo_exactSolRDSystemsAndNLWaveEqu_2001} or by N. Kudryashov and A. Zakharchenko in \cite{KudryashovZakharchenko_exactSolutionsLV_2015} (this latter concerning only the monostable case), but always under restrictions on the parameters, and in particular for specific values of the diffusion coefficients (not encompassing a full interval of values for the diffusion coefficients in \cref{syst_lv_u}).  

Our strategy will therefore be to assume that the parameter $\mu$ is small and use singular perturbation arguments to get a first order approximation (in terms of $\mu$) for the standing front (\vref{lem:exist_f_lv}). This will lead to a first order approximation (still in terms of $\mu$) for the quantity $\bar{c}$ we are interested in and in particular to its sign that will turn out to be positive (\vref{prop:bar_c_lv}). Before stating and proving these approximations, some notation is required. 
\subsection{Notation}
\label{subsec:not}
\begin{enumerate}
\item The estimates that will be computed in the remaining of this \namecref{subsec:not} will often involve the (small) quantity $\sqrt{\mu}$ (instead of $\mu$ itself). For this reason it will be convenient to have a specific notation for this quantity. Let us write:
\begin{equation}
\label{not_eps_mu}
\varepsilon = \sqrt{\mu}
\,.
\end{equation}
Note that this quantity $\varepsilon$ has nothing to do with the quantity $\epsilon$ introduced in \cref{sec:general} and displayed on \cref{fig:front_toy} (that one will not be used any more in the remaining of the paper).
\item The standing front and all related functions (for instance eigenfunctions) will turn out to depend slowly on the space variable $x$. For this purpose, it will sometimes be convenient to view them as functions of the space variable $y$ related to $x$ by:
\begin{equation}
\label{not_y_x}
y = \varepsilon x \Longleftrightarrow \frac{y}{\varepsilon} = x
\,.
\end{equation}
\item Up to an appropriate scaling, the standing front will be given at first order by the function 
\begin{equation}
\label{not_theta}
\theta:y\mapsto \tanh\Bigl(\frac{y}{2}\Bigr)
\,.
\end{equation}
(already encountered in the toy example of \cref{sec:toy}). This function is a solution of equation:
\[
\theta'' + \frac{1}{2}\theta(1-\theta^2)=0
\]
and the first order expansion along $\theta$ of this equation reads: $\ell \varphi =0$, where $\ell$ is the differential operator:
\begin{equation}
\label{not_ell}
\ell:\varphi\mapsto\varphi'' + \frac{1}{2}(1-3\theta^2)\varphi
\end{equation}
(compare with operator $\ell_{\mu}$ defined in \vref{lin_toy_v2}).
\item A notation is needed to deal with the remaining ``higher order terms'' (in $\varepsilon$) that will appear in the next computations. These higher order terms are slightly more involved than just real quantities. They depend on $\varepsilon$ and $x$, they vary slowly with respect to $x$, and they approach zero at an exponential rate when $x$ approaches $\pm\infty$. This approach to zero at infinity is important since integrals over the whole real line or half-real line will be made on various occasions. Let us define the space $\rrr$ of ``remaining terms'' as follows. A function
\[
r:(y,\varepsilon)\mapsto r(y,\varepsilon)
\]
belong to the set $\rrr$ if there exists a positive quantity $\delta$ such that:
\begin{itemize}
\item $r$ is defined and smooth on $\rr\times[0,\delta]$, 
\item for every integer $p$, the quantity 
\[
\sup_{(y,\varepsilon)\in\rr\times[0,\delta]} e^{\abs{y/2}}\abs{\partial_y^p r(y,\varepsilon)}
\]
is finite. 
\end{itemize}
\end{enumerate}
\subsection{Approximation of the standing front}
\label{subsec:exist_f_lv}
The following lemma makes use of the notation $\theta(\cdot)$ and $\rrr$ introduced above. Existence and uniqueness are stated in this lemma since they will be recovered automatically by the singular perturbation approach, but as mentioned above these results are well known,  \cite{Gardner_exisStabVWDegreeTheoreticApproach_1982,ConleyGardner_applicationGeneralizedMorseIndexTW_1984,Kan-on_parameterDependencePropSpeedTW_1995}. Therefore the main interest of this lemma is the approximation of the standing front that it provides. 
\begin{lemma}[existence of the standing front]
\label{lem:exist_f_lv}
For every positive and sufficiently small \\
quantity $\varepsilon$, the Lotka--Volterra system \cref{syst_lv_v} (written in transversal-longitudinal coordinates and where $\mu$ equals $\varepsilon^2$ --- see notation \cref{not_eps_mu}) admits a unique standing front 
\[
\phi_\varepsilon: \rr\rightarrow\rr^2, \quad x\mapsto \phi_\varepsilon(x) = \bigl(\phi_{\varepsilon,T}(x),\phi_{\varepsilon,L}(x)\bigr)
\]
connecting $E_1$ to $E_2$, taking its values in the first quadrant for the canonical coordinates (in other words such that $\phi_{\varepsilon,T}(x)$ is larger than $\abs{\phi_{\varepsilon,L}(x)}$ for every $x$ in $\rr$), and symmetric in the sense that:
\begin{itemize}
\item the fist component $x\mapsto\phi_{\varepsilon,T}(x)$ is even,
\item and the second component $x\mapsto\phi_{\varepsilon,L}(x)$ is odd.
\end{itemize}
In addition, there exist functions $r_T$ and $r_L$ in $\rrr$ such that, provided that $\varepsilon$ is small enough, for every real quantity $x$,
\[
\begin{aligned}
\phi_{\varepsilon,T}(x) &= 1 - \frac{\varepsilon^2}{2}\bigl(1-\theta(\varepsilon x)^2\bigl) + \varepsilon^3 r_T(\varepsilon x,\varepsilon)
\,,\\
\mbox{and}\quad\phi_{\varepsilon,L}(x) &= \theta(\varepsilon x) + \varepsilon r_L(\varepsilon x,\varepsilon)
\,.
\end{aligned}
\]
\end{lemma}
\begin{remark}
This lemma probably remains true without the additional constraint that the front must lie in the first quadrant for the canonical coordinates, but the formulation above is sufficient for our purpose. 
\end{remark}
\begin{proof}
Let $\varepsilon$ denote a (small) positive quantity. Replacing $\mu$ by $\varepsilon^2$, system \vref{syst_front_lv} governing stationary solutions of system \vref{syst_lv_v} reads:
\begin{equation}
\label{syst_front_lv_copy}
G_\mu(\phi)+\phi''=0
\Longleftrightarrow 
\left\{
\begin{aligned}
\phi_T'' &= \phi_T^2-\phi_T+\frac{\varepsilon^2}{2}(\phi_T^2-\phi_L^2) \\ 
\phi_L'' &= \phi_L(\phi_T-1)\,.
\end{aligned}
\right.
\end{equation}
The following intermediate lemma provides a priori bounds on the transversal component of the solution we are looking for. It is illustrated by \vref{fig:het_con}.
\begin{lemma}[a priori bound on the standing front]
\label{lem:a_priori_bound_phi_T}
Every global solution \\
$x\mapsto\bigl(\phi_T(x),\phi_L(x)\bigr)$ of system \cref{syst_front_lv_copy} connecting $E_1$ to $E_2$ and such that $\abs{\phi_L(\cdot)}$ is everywhere smaller than $\phi_T(\cdot)$ satisfies, for every real quantity $x$, 
\begin{equation}
\label{a_priori_bound_phi_T}
1-\frac{\varepsilon^2}{2} < \phi_T(x) < 1
\,.
\end{equation}
\end{lemma}
\begin{proof}[Proof of \cref{lem:a_priori_bound_phi_T}]
According to the first equation of system \cref{syst_front_lv_copy}, every such solution satisfies the differential inequalities
\[
\phi_T(\phi_T-1) \le \phi_T'' \le \phi_T\bigl( (1+\varepsilon^2/2) \phi_T - 1 \bigr)
\,.
\]
Since $\phi_T(x)$ must approach $1$ when $x$ approaches $\pm\infty$, it follows from the left-hand inequality that, for every real quantity $x$, 
\[
\phi_T(x) < 1
\,,
\]
and from the right-hand inequality that, for every real quantity $x$, 
\[
(1+\varepsilon^2/2) \phi_T(x)> 1
\,.
\]
Inequalities \cref{a_priori_bound_phi_T} follow. \Cref{lem:a_priori_bound_phi_T} is proved. 
\end{proof}
Let us pursue the proof of \cref{lem:exist_f_lv}.
According to \cref{lem:a_priori_bound_phi_T} it is natural to express system \cref{syst_front_lv_copy} in terms of the function $\eta_T$ defined by:
\[
\phi_T=1+\varepsilon^2\eta_T
\,.
\]
With this notation system \cref{syst_front_lv_copy} becomes
\[
\left\{
\begin{aligned}
\eta_T'' &= \eta_T+\frac{1}{2}(1-\phi_L^2)+\varepsilon^2\bigl(\eta_T+\eta_T^2(1+\varepsilon^2/2)\bigr) \\  
\phi_L'' &= \varepsilon^2\eta_T\phi_L 
\end{aligned}
\right.
\]
Using the notation
\[
\tilde\eta_T=\eta_T'
\quad\mbox{and}\quad
\tilde\phi_L=\frac{1}{\varepsilon}\phi'_L
\,,
\]
the previous system becomes
\begin{equation}
\label{front_lv_dim4}
\left\{
\begin{aligned}
\eta_T' &= \tilde\eta_T \\ 
\tilde\eta_T' &= \eta_T+\frac{1}{2}(1-\phi_L^2)+\varepsilon^2 \bigl(\eta_T+\eta_T^2(1+\varepsilon^2 /2)\bigr) \\  
\phi_L' &= \varepsilon\tilde\phi_L \\ 
\tilde\phi_L' &= \varepsilon\eta_T\phi_L
\end{aligned}
\right.
\end{equation}
This system is appropriate for a singular perturbation argument. It converges when $\varepsilon$ approaches $0$ to the ``fast'' system
\begin{equation}
\label{front_lv_fast}
\left\{
\begin{aligned}
\eta_T' &= \tilde\eta_T \\ 
\tilde\eta_T' &= \eta_T+\frac{1}{2}(1-\phi_L^2) \\  
\phi_L' &= 0 \\ 
\tilde\phi_L' &= 0
\end{aligned}
\right.
\end{equation}
for which the two-dimensional set
\[
\Sigma_0=\bigl\{(\eta_T,\tilde\eta_T,\phi_L,\tilde\phi_L)\in\rr^4: \eta_T=-\frac{1}{2}(1-\phi_L^2),\quad \tilde\eta_T=0\bigr\}
\]
is entirely made of equilibrium points. The matrix of the linearisation of the ``fast'' system \cref{front_lv_fast} at every point of $\Sigma_0$ reads: 
\[
\begin{pmatrix}
0&1&0&0\\  1&0&-\phi_L&0\\  0&0&0&0\\  0&0&0&0
\end{pmatrix}
.
\]
Its eigenvalues are $-1$, $+1$, and zero with multiplicity two. Thus the dynamics of the fast system~\cref{front_lv_fast} is ``transversely hyperbolic'' at every point of the equilibrium manifold $\Sigma_0$. This is the required hypothesis to apply the singular perturbation machinery. The set $\Sigma_0$ is the graph of the function 
\[
H_0:\rr^2\longrightarrow\rr^2, \quad (\phi_L,\tilde\phi_L)\longmapsto(\eta_T,\tilde\eta_T) = \Bigl(-\frac{1}{2}(1-\phi_L^2),0\Bigr)
\,.
\]
Let us consider the following subset of $\rr^2$:
\begin{equation}
\label{def_of_domain_D_lv}
D = \dd(0,2) = \bigl\{ (\phi_L,\tilde\phi_L)\in\rr^2 : \phi_L^2 + \tilde\phi_L^2 \le 4 \bigr\}
\,.
\end{equation}
We are going to apply Fenichel's global center manifold theorem \cite{Fenichel_geomSingPert_1979,Jones_geometricSingularPerturbationTheory_1995,Kaper_introductionGeometricMethodsSingularPerturbation_1999} with this set $D$ as definition set of the maps provided by this theorem.
The properties of this set that will be used are:
\begin{itemize}
\item it is compact, simply connected, with a smooth boundary,
\item its interior contains the trajectories of the heteroclinic connections \\ $y\mapsto\bigl(\theta(y),\pm\theta'(y)\bigr)$ (see \cref{fig:het_con}).
\end{itemize}
According to Fenichel's global center manifold theore, for every $\varepsilon$ sufficiently close to zero, there exists a map 
\[
H_\varepsilon: D\rightarrow\rr^2
\]
such that the graph of $H_\varepsilon$ (denoted by $\Sigma_\varepsilon$) is locally invariant under the dynamics of~\cref{front_lv_dim4}; this means that a solution with an initial condition on $\Sigma_\varepsilon$ remains on $\Sigma_\varepsilon$ as long as $(\phi_L,\tilde\phi_L)$ remains in $D$. Moreover, the map $H_\varepsilon$ coincides when $\varepsilon$ equals zero with the previous definition of $H_0$, and $H_\varepsilon$ depends smoothly on $\varepsilon$. Thus there exist smooth functions $h$ and $\tilde h$ of three variables, defined in a neighbourhood of $D\times\{0\}$ in $\rr^3$, such that for every $(\phi_L,\tilde\phi_L)$ in $D$ and $\varepsilon$ sufficiently small, 
\[
H_\varepsilon(\phi_L,\tilde\phi_L)=\Bigl(-\frac{1}{2}(1-\phi_L^2)+\varepsilon h(\phi_L,\tilde\phi_L,\varepsilon), \varepsilon \tilde h(\phi_L,\tilde\phi_L,\varepsilon)\Bigr)
\,.
\]
To study the ``slow'' dynamics in system~\cref{front_lv_dim4}, it is convenient to introduce some notation. 
According to \cref{not_y_x}, let us write:
\[
y=\varepsilon x \ \Leftrightarrow\  x = y/\varepsilon
\quad\mbox{and}\quad
\Phi_L(y) = \phi_L(y/\varepsilon)
\ \Leftrightarrow\ 
\Phi_L(\varepsilon x ) = \phi_L(x)
\,.
\]
With this notation, the two last equations of system~\cref{front_lv_dim4} reduce to: 
\[
\Phi_L''=\eta_T\Phi_L
\]
thus the law governing the dynamics of system~\cref{front_lv_dim4} on the ``slow'' manifold $\Sigma_\varepsilon$ reduces to: 
\begin{equation}
\label{f_lv_slow}
\Phi_L'' = -\frac{1}{2}\Phi_L(1-\Phi_L^2)+\varepsilon \Phi_L h(\Phi_L,\Phi_L',\varepsilon )
\,.
\end{equation}
At the limit $\varepsilon$ equals $0$, this equation becomes
\begin{equation}
\label{f_lv_slow_lim}
\Phi_L'' = -\frac{1}{2}\Phi_L(1-\Phi_L^2)
\,.
\end{equation}
The asymptotic equation~\cref{f_lv_slow_lim} admits two hyperbolic equilibria 
\[
(\Phi_L,\Phi_L')=(- 1,0)
\quad\mbox{and}\quad
(\Phi_L,\Phi_L')=(1,0)
\]
and two heteroclinic solutions connecting them, which are explicitly given by: 
\begin{equation}
\label{het_con_lv}
y\mapsto\pm\tanh\Bigl(\frac{y}{2}\Bigr)=\pm\theta(y)
\end{equation}
(see \cref{fig:het_con}).
\begin{figure}[!htbp]
	\centering
    \includegraphics[width=\textwidth]{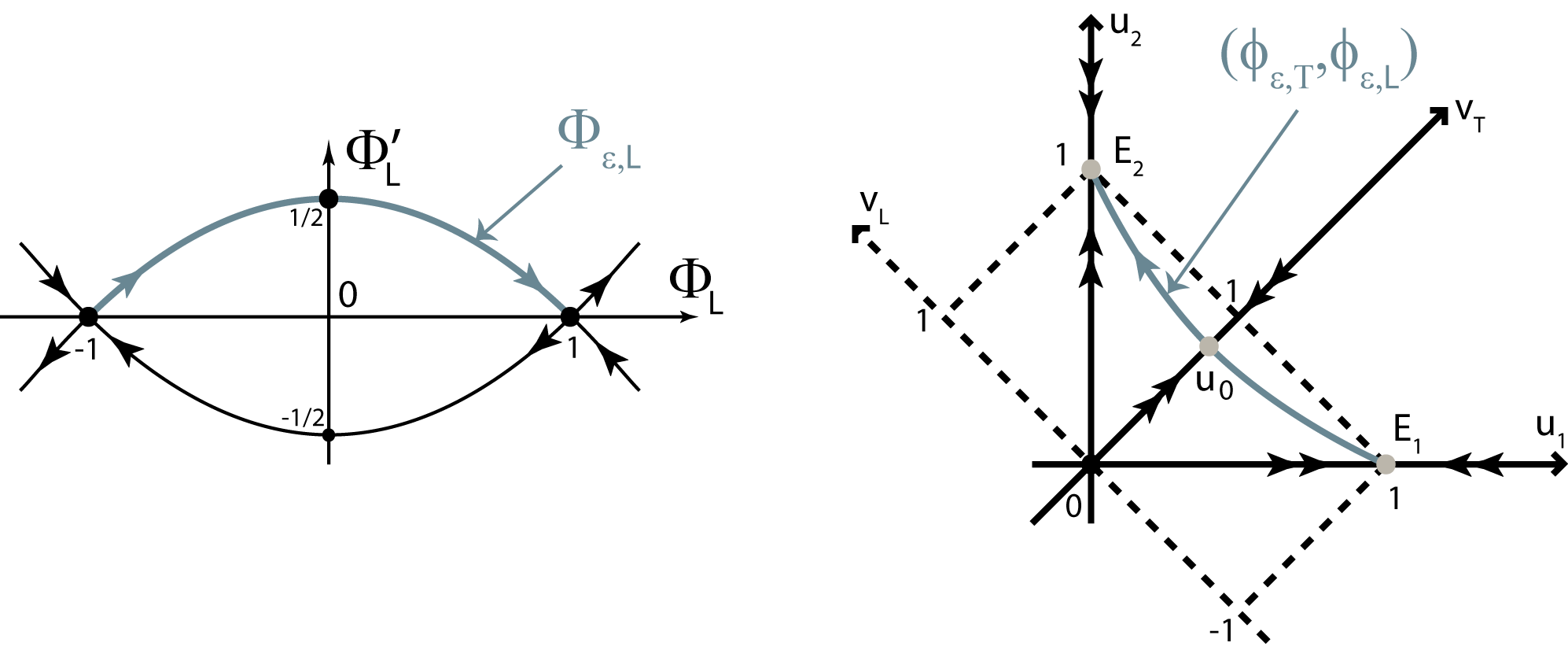}
    \caption{Heteroclinic connections in the phase space of equation~\cref{f_lv_slow} or~\cref{f_lv_slow_lim}, and corresponding standing front for the considered system.}
    \label{fig:het_con}
\end{figure}
Using the symmetries of the ``full'' system \cref{front_lv_dim4}, we are going to prove that these two heteroclinic connections persist for the (perturbed) reduced equation \cref{f_lv_slow} and remain symmetric with respect to the $\phi_L\leftrightarrow -\phi_L$ symmetry inherited from the $u_1\leftrightarrow u_2$ symmetry of the initial system. 

First, let us observe that for every sufficiently small positive quantity $\varepsilon$, the (perturbed) equation \cref{f_lv_slow} must admit two hyperbolic equilibria $(E_{-,\varepsilon},0)$ and $(E_{+,\varepsilon},0)$, with $E_{-,\varepsilon}$ close to $-1$ and $E_{+,\varepsilon}$ close to $1$. Let us mention here that, as usual with central manifolds, the slow manifold $\Sigma_\varepsilon$ is not necessarily unique, but it must contain every trajectory that remains globally in a small neighbourhood of it (\cite{Fenichel_geomSingPert_1979,Jones_geometricSingularPerturbationTheory_1995,Kaper_introductionGeometricMethodsSingularPerturbation_1999}). Therefore it must contain the equilibria corresponding to $E_1$ and $E_2$. It follows that $E_{-,\varepsilon}-=-1$ and $E_{+,\varepsilon}=+1$, in other words:
\begin{equation}
\label{equil_glob_cent_lv}
h(-1,0,\varepsilon)= h(-1,0,\varepsilon) =0
\,.
\end{equation}

Now, the robustness of the heteroclinic connections \cref{het_con_lv} is asserted by the following intermediate lemma.
\begin{lemma}[robustness of heteroclinic connections]
\label{lem:robust_het_con}
For every sufficiently small positive quantity $\varepsilon$, there exists a global solution
\[
y\mapsto\Phi_{\varepsilon,L}(y)
\]
of the reduced equation \cref{f_lv_slow} such that
\[
\Phi_{\varepsilon,L}(y) \rightarrow -1
\quad\mbox{when}\quad
y\rightarrow -\infty
\quad\mbox{and}\quad
\Phi_{\varepsilon,L}(y) \rightarrow +1
\quad\mbox{when}\quad
y\rightarrow +\infty 
\,, 
\]
and, for every real quantity $y$,
\[
\Phi_{\varepsilon,L}(-y) = -\Phi_{\varepsilon,L}(y)
\,.
\]
\end{lemma}
\begin{proof}
The ``full'' system \cref{front_lv_dim4} admits two symmetries, the reversibility $x\leftrightarrow -x$ and the $\phi_L\leftrightarrow -\phi_L$ symmetry inherited from the $u_1\leftrightarrow u_2$ symmetry of the initial system. To be more precise, according to these two symmetries, if 
\[
x\mapsto\bigl(\eta_T(x),\tilde\eta_T(x),\phi_L(x), \tilde\phi_L(x)\bigr)
\]
is a solution of system \cref{front_lv_dim4}, then the following two functions are also solutions:
\[
\begin{aligned}
& x\mapsto\bigl(\eta_T(-x),-\tilde\eta_T(-x),\phi_L(-x), -\tilde\phi_L(-x)\bigr) \\
\mbox{and}\quad
& x\mapsto\bigl(\eta_T(x),\tilde\eta_T(x),-\phi_L(x), -\tilde\phi_L(x)\bigr)
\end{aligned}
\]
It is well known that \emph{local} center manifolds of systems admitting equivariant or reversibility symmetries can be chosen in such a way that those manifolds be themselves invariant under these symmetries --- note that since center manifolds are not necessarily unique this is however not obvious --- and as a consequence in such a way that the reduced systems (obtained by reduction of the initial systems to those symmetric local center manifolds) still admit the same symmetries as the initial system. See \cite{Ruelle_bifurcationPresenceSymmetryGroup_1973} and \cite{AdelmeyerIooss_topicsBifurcationTheoryApplications_1992,HaragusIooss_localBifurcationsInfiniteDim_2011} for more recent expositions, the last one concerning infinite dimensional dynamical systems. If a similar result could be invoked for \emph{global} center manifolds, we would be able to choose our global center manifold $\Sigma_\varepsilon$ in such a way that it is invariant under the two symmetries of the full system \cref{front_lv_dim4}, namely in such a way that:
\begin{itemize}
\item $h(\phi_L,\tilde\phi_L,\varepsilon)$ is even with respect to $\phi_L$, 
\item $h(\phi_L,\tilde\phi_L,\varepsilon)$ is even with respect to $\tilde\phi_L$, 
\item $\tilde h(\phi_L,\tilde\phi_L,\varepsilon)$ is odd with respect to $\phi_L$,
\item $\tilde h(\phi_L,\tilde\phi_L,\varepsilon)$ is odd with respect to $\tilde\phi_L$,
\end{itemize}
and as a consequence the reduced equation \cref{f_lv_slow} would admit the same two symmetries, and the conclusions of \cref{lem:robust_het_con} would immediately follow from these symmetries. Unfortunately, to the knowledge of the author, no statement concerning the existence of global center manifolds satisfying reversibility and equivariant symmetries and applicable in our case is available in the existing litterature. However we are going to recover this symmetry for the aforementioned heteroclinic connections by another (less direct) argument. 
\begin{figure}[!htbp]
	\centering
    \includegraphics[width=0.6\textwidth]{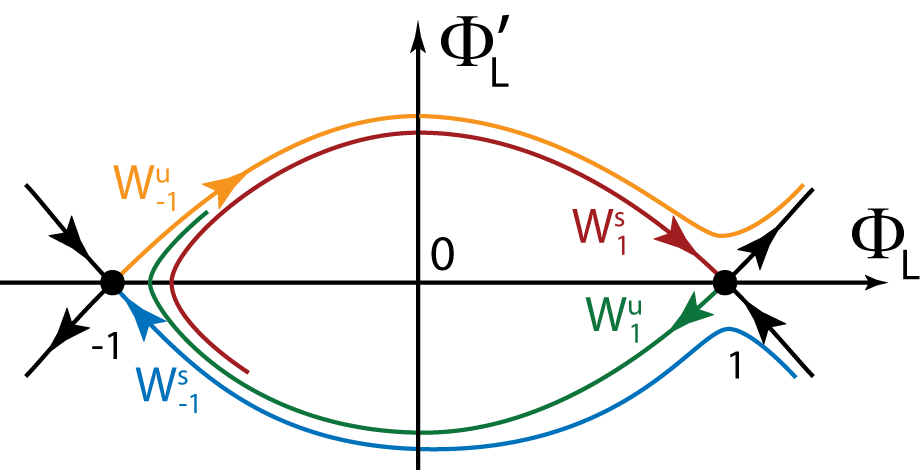}
    \caption{The four trajectories $\WuMinusOne$ and $\WsOne$ and $\WuOne$ and $\WsMinusOne$. As argued in the proof of \cref{lem:robust_het_con}, if these four trajectories differ, one of them (in this case $\WsOne$) is ``trapped'' by the other ones.}
    \label{fig:stab_unstab_man}
\end{figure}

Let us fix $\varepsilon$ (positive, small) and let us consider the four trajectories $\WuMinusOne$ and $\WsOne$ and $\WuOne$ and $\WsMinusOne$ depicted on \cref{fig:stab_unstab_man} (they are part of the stable and unstable manifolds of $(-1,0)$ and $(1,0)$ for the reduced equation \cref{f_lv_slow} for this value of $\varepsilon$). Let us proceed by contradiction and assume that $\WuMinusOne$ and $\WsOne$ do not coincide and that $\WuOne$ and $\WsMinusOne$ do not coincide. Then, by a Jordan curve argument (see \cref{fig:stab_unstab_man}), we see that at least one of those four trajectories does remains ``trapped'' (by the three others) in the domain $D$ defined in \cref{def_of_domain_D_lv}. Let 
\[
x\mapsto\bigl(\eta_T(x),\tilde\eta_T(x),\phi_L(x), \tilde\phi_L(x)\bigr)
\]
denote a solution of the full system \cref{front_lv_dim4} corresponding to this trajectory. Then, due to the symmetries of this full system \cref{front_lv_dim4}, the three functions
\[
\begin{aligned}
& x\mapsto\bigl(\eta_T(-x),-\tilde\eta_T(-x),\phi_L(-x), -\tilde\phi_L(-x)\bigr) \\
\mbox{and}\quad
& x\mapsto\bigl(\eta_T(x),\tilde\eta_T(x),-\phi_L(x), -\tilde\phi_L(x)\bigr) \\
\mbox{and}\quad
& x\mapsto\bigl(\eta_T(-x),-\tilde\eta_T(-x),-\phi_L(-x), \tilde\phi_L(-x)\bigr)
\end{aligned}
\]
are still solutions of the same system, and these three additional solutions still globally remain in a small neighbourhood of the center manifold $\Sigma_\varepsilon$. As a consequence, theses three additional solutions must also belong to $\Sigma_\varepsilon$, leading to a topological contradiction, see \cref{fig:stab_unstab_man}. 

Thus at least one among the two pairs $(\WuMinusOne,\WsOne)$ and $(\WuOne,\WsMinusOne)$ must be reduced to a single trajectory, and by a similar argument this must actually be the case for both pairs. This proves the existence of the two heteroclinic connections. Their symmetries follows from the same argument. \Cref{lem:robust_het_con} is proved.
\end{proof}
Let us define the function $r_L:(y,\varepsilon)\mapsto r_L(y,\varepsilon)$ by:
\[
\Phi_{\varepsilon,L}(y)=\theta(y)+\varepsilon r_L(y,\varepsilon)
\,.
\]
Since the eigenvalues of equilibria $(-1,0)$ and $(1,0)$ of equation \cref{f_lv_slow} are close to $-1$ and $+1$, the ``remaining'' function $r_L$ belongs to the space $\rrr$ defined in \cref{subsec:not}. 
Let us define the function
\[
\phi_\varepsilon:\rr\rightarrow\rr^2,
\quad
x\mapsto\bigl(\phi_{\varepsilon,T}(x),\phi_{\varepsilon,L}(x)\bigr)
\]
by:
\begin{equation}
\label{expr_phi}
\left\{
\begin{aligned}
\phi_{\varepsilon,T}(x) 
&= 1 - \frac{\varepsilon^2}{2}(1-\phi_{\varepsilon,L}^2) + \varepsilon^3 h\bigl( \Phi_{\varepsilon,L}(\varepsilon x) , \Phi_{\varepsilon,L}'(\varepsilon x),\varepsilon \bigr) 
\,, \\
\phi_{\varepsilon,L}(x) 
&= \Phi_{\varepsilon,L}(\varepsilon x) 
\,.
\end{aligned}
\right.
\end{equation}
This function is a standing front connecting $E_1$ to $E_2$, it is $u_1\leftrightarrow u_2$-symmetric (in other words $\phi_\varepsilon(-x)$ equals $\sss\phi_\varepsilon(x)$ for every $x$ in $\rr$), and it takes its values in the ``first quadrant'' $\phi_T>\abs{\phi_L}$. In addition, if we define the ``remaining'' function $r_T$ by:
\[
\phi_{\varepsilon,T}(x) = 1 - \frac{\varepsilon^2}{2}\bigl(1-\theta(\varepsilon x)^2\bigr) + \varepsilon^3 r_T(\varepsilon x,\varepsilon)
\]
then, according to equalities \cref{equil_glob_cent_lv}, this function $r_T$ belongs to $\rrr$. The proof of \cref{lem:exist_f_lv} is thus complete.
\end{proof}
\subsection{First-order variation of the front speed}
\label{subsec:1rst_ord}
With the notation of \cref{sec:general}, the diffusion matrix $\ddd$ equals identity. Let $\varepsilon$ denote a positive quantity, sufficiently small so that \vref{lem:exist_f_lv} holds, and let us consider the standing front $\phi_\varepsilon(\cdot)$ provided by this lemma. 
\begin{lemma}[spectral stability of the standing front]
\label{lem:lin_stab_lv}
{\ }\\
The standing front
\[
x\mapsto\phi_\varepsilon (x)
\]
is spectrally stable (in the sense of \vref{subsec:spec_stab_gen}, that is including the fact that the eigenvalue $0$ has an algebraic multiplicity equal to $1$) for the Lotka--Volterra system \cref{syst_lv_v}. 
\end{lemma}
\begin{proof}
As mentioned in \cref{subsec:lv_standing_front}, this follows from the general stability results proved by Kan-on and Fang in \cite{Kan-on_parameterDependencePropSpeedTW_1995,Kan-onFang_stabilityMonotoneTW_1996}.
\end{proof}
Let us choose the perturbation matrix $\bar{\ddd}$ exactly as in the toy example considerer in \cref{sec:toy}, see definition \vref{def_bar_D_toy}.
For those items, all the hypotheses (H1-4) of \cref{sec:general} are satisfied. Let us denote by
\[
\mathcal{L}_\varepsilon
\quad\mbox{and}\quad
\mathcal{L}^*_\varepsilon
\quad\mbox{and}\quad
\psi_\varepsilon
\quad\mbox{and}\quad
\bar c_\varepsilon
\]
the objects that were denoted by $\mathcal{L}$ and $\mathcal{L}^*$ and $\psi$ and $\bar c$ in \cref{sec:general} (these objects now depend on $\varepsilon$), and let us again denote by $\theta$ the function: $y\mapsto \tanh(y/2)$. The aim of this \namecref{subsec:1rst_ord} is to prove the following proposition.
\begin{proposition}[first order variation of front speed]
\label{prop:bar_c_lv}
The \\
following estimate holds:
\begin{equation}
\label{bar_c_lv}
\bar{c}_\varepsilon \ \sim_{\varepsilon\rightarrow 0^+}\ -\varepsilon\ \frac{\int_0^{+\infty}\theta(y)\theta'(y)\theta''(y)\,dy}{\norm{\theta'}_{L^2(\rr_+,\rr)}^2}
\,.
\end{equation}
\end{proposition}
As a consequence the quantity $\bar{c}$ is positive for every sufficiently small positive quantity $\varepsilon$. In other words, for the Lotka--Volterra competition model in the bistable regime, close to the onset of bistability, an increase of mobility provides an advantage. Since in this case the competition between the two species can be qualified as ``moderately strong'', we recover the interpretation given for the toy example in \cref{sec:toy}, that is the fact that when competition is moderately strong an increase of mobility is advantageous. 
\begin{proof}
We are going to use expression \cref{solv_cond} of $\bar c$, namely, according to the expression of $\bar\ddd$,
\begin{equation}
\label{solv_cond_lv}
\bar c_\varepsilon = \int_{-\infty}^{+\infty} \bigl(\psi_{\varepsilon,T}(x)\phi_{\varepsilon,L}''(x) + \psi_{\varepsilon,L}(x)\phi_{\varepsilon,T}''(x)\bigl)\, dx
\end{equation}
where $x\mapsto\psi_\varepsilon(x)=\bigl(\psi_{\varepsilon,T}(x),\psi_{\varepsilon,L}(x)\bigr)$ (written in the $v$-coordinate system) is the solution of 
$\mathcal{L}^*_\varepsilon\psi=0$ satisfying the normalization condition \cref{norm_cond}, namely:
\begin{equation}
\label{norm_cond_lv}
\langle\psi_\varepsilon,\phi_\varepsilon'\rangle_{L^2(\rr,\rr^2)} = 1
\end{equation}
(here there would be no gain from restricting the integrals in \cref{solv_cond_lv} to $\rr_+$ as in the reduced expression \vref{s_cond_rr_plus}). 

\Cref{lem:exist_f_lv} provides convenient approximations for the functions $\phi_{\varepsilon,T}''(\cdot)$ and $\phi_{\varepsilon,L}''(\cdot)$, thus what remains to be done is to get similar approximations for $\psi_{\varepsilon,T}(\cdot)$ and $\psi_{\varepsilon,L}(\cdot)$. According to expression \cref{syst_lv_v} of the Lotka--Volterra system in the transversal-longitudinal coordinate systems, the operator $\mathcal{L}_\varepsilon$ reads (in the same coordinate system):
\[
\mathcal{L}_\varepsilon
\begin{pmatrix}
\varphi_T \\ \varphi_L
\end{pmatrix}
(x)
= 
\begin{pmatrix}
1-(2+\varepsilon^2)\phi_{\varepsilon,T}(x) & \varepsilon^2\phi_{\varepsilon,L}(x) \\
-\phi_{\varepsilon,L}(x) & 1-\phi_{\varepsilon,T}(x)
\end{pmatrix} 
\begin{pmatrix}
\varphi_T(x) \\ \varphi_L(x)
\end{pmatrix}
+
\begin{pmatrix}
\varphi_T''(x) \\ \varphi_L''(x)
\end{pmatrix}
\,.
\]
\begin{notation}
For the remaining of \cref{sec:lv}, let us use the notation $r(\cdot,\cdot)$ to denote every function in the space $\rrr$ defined in \cref{subsec:not}, or every $2\times2$ or $4\times 4$ matrix having all its coefficients in this space $\rrr$. Thus each of the symbols $r(\cdot, \cdot)$ that appear in the expressions below corresponds to a (different) element of $\rrr$ or matrix of elements of $\rrr$. 
\end{notation}
With this notation and according to the approximation provided by \cref{lem:exist_f_lv}, the expression above reduces to
\[
\mathcal{L}_\varepsilon 
\begin{pmatrix}
\varphi_T \\ \varphi_L
\end{pmatrix}
(x)
= 
\begin{pmatrix}
-1+\varepsilon^2 r(y, \varepsilon) & \varepsilon^2 \theta(y) + \varepsilon^3 r(y, \varepsilon) \\
-\theta(y) + \varepsilon r(y, \varepsilon) & \frac{\varepsilon^2}{2}\bigl(1-\theta(y)^2\bigr)+\varepsilon^3 r(y, \varepsilon)
\end{pmatrix} 
\begin{pmatrix}
\varphi_T (x)\\ \varphi_L(x)
\end{pmatrix}
+
\begin{pmatrix}
\varphi_T''(x) \\ \varphi_L''(x)
\end{pmatrix}
,.
\]
According to this expression the system $\mathcal{L}^*_\varepsilon\psi=0$ reads
\[
\begin{pmatrix}
\psi_{\varepsilon,T}''(x) \\ \psi_{\varepsilon,L}''(x)
\end{pmatrix}
=
\begin{pmatrix}
1 + \varepsilon^2 r(y,\varepsilon) & \theta(y) + \varepsilon r(y,\varepsilon) \\
-\varepsilon^2 \theta(y) + \varepsilon^3 r(y,\varepsilon) & -\frac{\varepsilon^2}{2}\bigl(1-\theta(y)^2\bigr) + \varepsilon^3 r(y,\varepsilon)
\end{pmatrix}
\begin{pmatrix}
\psi_{\varepsilon,T}(x) \\ \psi_{\varepsilon,L} (x)
\end{pmatrix}
\]
or equivalently
\begin{equation}
\label{syst_psi_lv}
\begin{pmatrix}
\psi_{\varepsilon,T}''(x) \\ \frac{1}{\varepsilon^2}\psi_{\varepsilon,L} ''(x)
\end{pmatrix}
= 
\begin{pmatrix}
1 & \theta(y)  \\
- \theta(y)  & -\frac{1}{2}\bigl(1-\theta(y)^2\bigr) 
\end{pmatrix}
\begin{pmatrix}
\psi_{\varepsilon,T}(x) \\ \psi_{\varepsilon,L}(x) 
\end{pmatrix}
+
\varepsilon r(y,\varepsilon)
\begin{pmatrix}
\psi_{\varepsilon,T}(x) \\ \psi_{\varepsilon,L} (x)
\end{pmatrix}
\,.
\end{equation}
Let us introduce the functions $\tilde\psi_{\varepsilon,T}(\cdot)$ and $\tilde\psi_{\varepsilon,L}(\cdot)$ defined by (for every real quantity $x$):
\[
\tilde\psi_{\varepsilon,T}(x) = \psi'_{\varepsilon,T}(x)
\quad\mbox{and}\quad
\tilde\psi_{\varepsilon,L}(x) = \frac{1}{\varepsilon}\psi_{\varepsilon,L}'(x)
\,.
\]
Then the previous system becomes
\[
\begin{pmatrix}
\psi_{\varepsilon,T}' \\ \tilde\psi_{\varepsilon,T}' \\ \psi_{\varepsilon,L}' \\ \tilde\psi_{\varepsilon,L}'
\end{pmatrix}
(x)
=
\begin{pmatrix}
0 & 1 & 0 & 0 \\
1 & 0 & \theta(y) & 0 \\
0 & 0 & 0 & 0 \\
0 & 0 & 0 & 0
\end{pmatrix}
\begin{pmatrix}
\psi_{\varepsilon,T} \\ \tilde\psi_{\varepsilon,T} \\ \psi_{\varepsilon,L} \\ \tilde\psi_{\varepsilon,L}
\end{pmatrix}
(x)
+
\varepsilon r(y,\varepsilon)
\begin{pmatrix}
\psi_{\varepsilon,T} \\ \tilde\psi_{\varepsilon,T} \\ \psi_{\varepsilon,L} \\ \tilde\psi_{\varepsilon,L}
\end{pmatrix}
(x)
\,.
\]
Another change of variables will fire the non-diagonal term in the $4\times4$ matrix above. For this purpose, let us introduce the functions $\eta_{\varepsilon,T}(\cdot)$ and $\tilde\eta_{\varepsilon,T}(\cdot)$ defined by (for every real quantity $x$):
\begin{equation}
\label{def_eta_lv}
\eta_{\varepsilon,T}(x) = \psi_{\varepsilon,T}(x) + \theta(y) \psi_{\varepsilon,L}(x)
\quad\mbox{and}\quad
\tilde\eta_{\varepsilon,T}(x) = \eta_{\varepsilon,T}'(x)
\,.
\end{equation}
Then,
\[
\tilde\eta_{\varepsilon,T}'(x) = \eta_{\varepsilon,T}''(x) = \tilde\psi_{\varepsilon,T}'(x) + \varepsilon^2\theta''(y)\psi_{\varepsilon,L}(x) + 2\varepsilon^2 \theta'(y)\tilde\psi_{\varepsilon,L}(x)+ \varepsilon\theta(y)\tilde\psi_{\varepsilon,L}'(x)
\,,
\]
thus, since according to the second line of the system above the dominant term in the expression of $\tilde\psi_{\varepsilon,T}'(x)$ is $\psi_{\varepsilon,T}(x)+\theta(y)\tilde\psi_{\varepsilon,T}(x)$ and since this dominant term equals $\eta_{\varepsilon,T}(x)$, this system can be rewritten as follows:
\[
\begin{pmatrix}
\eta_{\varepsilon,T}' \\ \tilde\eta_{\varepsilon,T}' \\ \psi_{\varepsilon,L}' \\ \tilde\psi_{\varepsilon,L}'
\end{pmatrix}
(x)
=
\begin{pmatrix}
0 & 1 & 0 & 0 \\
1 & 0 & 0 & 0 \\
0 & 0 & 0 & 0 \\
0 & 0 & 0 & 0
\end{pmatrix}
\begin{pmatrix}
\eta_{\varepsilon,T} \\ \tilde\eta_{\varepsilon,T} \\ \psi_{\varepsilon,L} \\ \tilde\psi_{\varepsilon,L}
\end{pmatrix}
(x)
+
\varepsilon r(y,\varepsilon)
\begin{pmatrix}
\eta_{\varepsilon,T} \\ \tilde\eta_{\varepsilon,T} \\ \psi_{\varepsilon,L} \\ \tilde\psi_{\varepsilon,L}
\end{pmatrix}
(x)
\,.
\]
Since the first $2\times2$ block of the $4\times4$ matrix of this system is hyperbolic, and since the quantities $\eta_{\varepsilon,T}(x)$ and $\tilde\eta_{\varepsilon,T}(x)$ and $\psi_{\varepsilon,L}(x)$ and $\tilde\psi_{\varepsilon,L}(x)$ must approach zero when $t$ approaches plus or minus infinity, this shows that there exist a positive quantity $C$, independent of $\varepsilon$ provided that $\varepsilon$ is sufficiently small, such that, for all $x$ in $\rr$, 
\begin{equation}
\label{bound_eta_lv}
\abs{\eta_{\varepsilon,T}(x)}\le C\varepsilon\bigl(\abs{\psi_{\varepsilon,L}(x)} + \abs{\tilde\psi_{\varepsilon,L}(x)}\bigr)
\quad\mbox{and}\quad
\abs{\tilde\eta_{\varepsilon,T}(x)}\le C\varepsilon\bigl(\abs{\psi_{\varepsilon,L}(x)} + \abs{\tilde\psi_{\varepsilon,L}(x)}\bigr)
\,.
\end{equation}
Let us introduce the function $\Psi_{\varepsilon,L}(\cdot)$ defined by (for every $(x,y)$ in $\rr^2$ with $y=\varepsilon x$):
\[
\Psi_{\varepsilon,L}(y) = \psi_{\varepsilon,L}\Bigl(\frac{y}{\varepsilon}\Bigr) 
\Leftrightarrow
\Psi_{\varepsilon,L}(\varepsilon x) = \psi_{\varepsilon,L}(x)
\,.
\]
With this notation, the second equation of system \cref{syst_psi_lv} becomes:
\[
\Psi_{\varepsilon,L}''(y) = -\theta(y)\psi_{\varepsilon,T}\Bigl(\frac{y}{\varepsilon}\Bigr) - \frac{1}{2}\bigl(1-\theta(y)^2\bigr)\Psi_{\varepsilon,L}(y) + \varepsilon r(y,\varepsilon) \Psi_{\varepsilon,L}(y)
\]
Thus, according to the notation \cref{def_eta_lv}, 
\[
\Psi_{\varepsilon,L}''(y) = \frac{1}{2}\bigl(3\theta(y)^2-1\bigr)\Psi_{\varepsilon,L}(y) -\theta(y)\eta_{\varepsilon,T}\Bigl(\frac{y}{\varepsilon}\Bigr) + \varepsilon r(y,\varepsilon) \Psi_{\varepsilon,L}(y)
\]
and thus, according to inequalities \cref{bound_eta_lv}, and up to increasing the quantity $C$, for all $y$ in $\rr$ (using the notation $\ell$ introduced in \cref{subsec:not}), 
\begin{equation}
\label{ell_psi_L_lv}
\abs{(\ell\Psi_{\varepsilon,L})(y)} = \abs{\Psi_{\varepsilon,L}''(y) - \frac{1}{2}\bigl(3\theta(y)^2-1\bigr) \Psi_{\varepsilon,L}(y)}\le C\varepsilon\bigl(\abs{\Psi_{\varepsilon,L}(y)} + \abs{\Psi_{\varepsilon,L}'(y)}\bigr)
\,.
\end{equation}
Let 
\[
\alpha = \frac{\langle\Psi_{\varepsilon,L},\theta'\rangle_{L^2(\rr,\rr)}}{\norm{\theta'}_{L^2(\rr,\rr)}^2}
\]
and, for all $y$ in $\rr$, let
\begin{equation}
\label{def_chi_L}
\chi_{\varepsilon,L}(y) = \Psi_{\varepsilon,L}(y)-\alpha\theta'(y)
\,.
\end{equation}
By construction, the function $\chi_{\varepsilon,L}(\cdot)$ is orthogonal to $\theta'$, that is to the kernel of $\ell$ and since $\ell\Psi_{\varepsilon,L}$ and $\ell\chi_{\varepsilon,L}$ are equal it follows from \cref{ell_psi_L_lv} that, for every real quantity $y$, 
\[
\abs{(\ell\chi_{\varepsilon,L})(y)}\le C\varepsilon\bigl(\abs{\Psi_{\varepsilon,L}(y)} + \abs{\Psi_{\varepsilon,L}'(y)}\bigr)
\,.
\]
As a consequence, up to increasing the quantity $C$, for every $y$ in $\rr$,
\begin{equation}
\label{bound_chi_L_chi_L_prime}
\abs{\chi_{\varepsilon,L}(y)} \le C\varepsilon\bigl(\abs{\Psi_{\varepsilon,L}(y)} + \abs{\Psi_{\varepsilon,L}'(y)}\bigr)
\quad\mbox{and}\quad
\abs{\chi_{\varepsilon,L}'(y)} \le C\varepsilon\bigl(\abs{\Psi_{\varepsilon,L}(y)} + \abs{\Psi_{\varepsilon,L}'(y)}\bigr)
\,.
\end{equation}
It follows from \cref{def_chi_L,bound_chi_L_chi_L_prime} that, for every real quantity $y$, 
\[
(1-C\varepsilon)\bigl( \abs{\Psi_{\varepsilon,L}(y)} + \abs{\Psi_{\varepsilon,L}'(y)} \bigr) \le \alpha \bigl( \abs{\theta'(y)} + \abs{\theta''(y)} \bigr)
\]
and as a consequence, provided that $\varepsilon$ is sufficiently small,
\begin{equation}
\label{bound_Psi_L_plus_Psi_L_prime}
\abs{\Psi_{\varepsilon,L}(y)} + \abs{\Psi_{\varepsilon,L}'(y)} = r(y,\varepsilon)
\,.
\end{equation}
As a consequence, it follows from \cref{def_chi_L,bound_chi_L_chi_L_prime} that
\begin{equation}
\label{approx_psi_L_psi_T}
\psi_{\varepsilon,L}(x) = \alpha\theta'(\varepsilon x) + \varepsilon r (\varepsilon x, \varepsilon)
\quad\mbox{and}\quad
\psi_{\varepsilon,L}'(x) = \varepsilon \alpha\theta'(\varepsilon x) + \varepsilon^2 r (\varepsilon x, \varepsilon)
\,.
\end{equation}
Besides, it follows from the upper bounds \cref{bound_eta_lv,bound_Psi_L_plus_Psi_L_prime} that
\[
\eta_{\varepsilon,T}(x) = \varepsilon r(\varepsilon x, \varepsilon)
\,.
\]
thus, according to the definition \cref{def_eta_lv} of $\eta_{\varepsilon,T}(\cdot)$, 
\begin{equation}
\label{approx_psi_T}
\psi_{\varepsilon,T}(x) = -\alpha \theta(\varepsilon x) \theta'(\varepsilon x) +  \varepsilon r(\varepsilon x, \varepsilon)
\,.
\end{equation}
The normalization condition \cref{norm_cond_lv} will provide the approximate value of the quantity $\alpha$. This normalization condition reads:
\[
\langle\psi_{\varepsilon,T},\phi_{\varepsilon,T}'\rangle_{L^2(\rr,\rr)} + \langle\psi_{\varepsilon,L}, \phi_{\varepsilon,L}'\rangle_{L^2(\rr,\rr)}=1
\,,
\]
in other words, according to the expressions of $\phi_{\varepsilon,T}(\cdot)$ and $\phi_{\varepsilon,L}(\cdot)$ provided by \vref{lem:exist_f_lv},
\[
\int_{-\infty}^{+\infty} \psi_{\varepsilon,T}(x) \cdot \varepsilon^3 r(\varepsilon x,\varepsilon)\, dx + \int_{-\infty}^{+\infty} \psi_{\varepsilon,L}(x)\cdot\bigl( \varepsilon \theta'(\varepsilon x) + \varepsilon^2 r(\varepsilon x, \varepsilon)\bigr) \, dx = 1
\,.
\]
According to the expression \cref{approx_psi_T}, the first integral of the left-hand side of this last inequality is a $\ooo_{\varepsilon\rightarrow0}(\varepsilon^2)$. Thus it follows from the expression \cref{approx_psi_L_psi_T} for $\psi_{\varepsilon,L}(\cdot)$ that:
\begin{equation}
\label{estim_norm_lv}
\alpha=\frac{1}{\norm{\theta'}_{L^2(\rr,\rr)}^2}+\ooo_{\varepsilon\rightarrow0}(\varepsilon)
\,.
\end{equation}
We are now in position to estimate the value of $\bar{c}_\varepsilon$ given by \cref{solv_cond_lv}. According to \cref{lem:exist_f_lv}, 
\[
\phi_{\varepsilon,T}''(x)=\varepsilon^4 r(\varepsilon x, \varepsilon)
\,,
\]
thus it follows from \cref{solv_cond_lv} and the expression \cref{approx_psi_L_psi_T} of $\psi_{\varepsilon,L}(\cdot)$ that
\[
\bar{c}_\varepsilon = \ooo(\varepsilon^3) + \int_{-\infty}^{+\infty}\, \psi_{\varepsilon,T}(x)\ \phi_{\varepsilon,L}''(x) \, dx  
\,,
\]
thus, according to the expression of $\phi_{\varepsilon,L}(\cdot)$ provided by \vref{lem:exist_f_lv} and the expression \cref{approx_psi_T} of $\psi_{\varepsilon,T}(\cdot)$,
\[
\begin{aligned}
\bar{c}_\varepsilon &= \ooo(\varepsilon^3) + \int_{-\infty}^{+\infty} \bigl( -\alpha \theta(\varepsilon x) \theta'(\varepsilon x) +  \varepsilon r(\varepsilon x, \varepsilon) \bigr)\cdot\bigl(\varepsilon^2 \theta''(\varepsilon x) + \varepsilon^3 r(\varepsilon x,\varepsilon)\bigr) \, dx \\
& = \ooo(\varepsilon^3) - \varepsilon^2\alpha \int_{-\infty}^{+\infty} \theta(\varepsilon x) \theta'(\varepsilon x)\theta''(\varepsilon x) \, dx \\
& = \ooo(\varepsilon^3) - \varepsilon\alpha\int_{-\infty}^{+\infty} \theta(y) \theta'(y)\theta''(y) \, dy
\,.
\end{aligned}
\]
Finally, according to the expression \cref{estim_norm_lv} for $\alpha$,
\[
\bar{c}_\varepsilon \sim_{\varepsilon\to 0} -\frac{\varepsilon}{\norm{\theta'}_{L^2(\rr,\rr)}^2} \int_{-\infty}^{+\infty} \theta(y)\theta'(y)\theta''(y)\,dy
\,,
\]
and restricting the integrals to $\rr_+$ estimate \cref{bar_c_lv} follows. \Cref{prop:bar_c_lv} is proved.
\end{proof}
\section{Appendix}
\label{sec:app}
\subsection{An elementary property of the solutions of a second order conservative equation}
\label{subsec:sol_forced_2nd_order}
Let $f:[0,+\infty)\rightarrow\rr$ denote a continuous function satisfying
\[
f(t)\rightarrow 0 
\quad\mbox{when}\quad 
t\rightarrow +\infty
\quad\mbox{and}\quad
f(\cdot) \mbox{ does not vanish on } (0,+\infty)
\,,
\]
and let us consider the following second order equation:
\begin{equation}
\label{equ_forced_2nd_order}
\ddot u = u + f 
\,.
\end{equation}
The aim of this \namecref{subsec:sol_forced_2nd_order} is to prove the following lemma.
\begin{lemma}[solution homoclinic to $0$]
\label{lem:sol_forced_2nd_order}
There exists a unique solution $t\mapsto u(t)$ of equation \cref{equ_forced_2nd_order} defined on $[0,+\infty)$ such that
\begin{equation}
\label{prop_sol_forced_2nd_order}
u(0) = 0
\quad\mbox{and}\quad
u(t)\rightarrow 0 
\quad\mbox{when}\quad 
t\rightarrow +\infty
\,.
\end{equation}
This solution does not vanish on $(0,+\infty)$, and its sign is opposite to the sign of $f(\cdot)$. 
\end{lemma}
\begin{proof}
Let $t\mapsto u(t)$ denote a solution of equation \cref{equ_forced_2nd_order} on $\rr_+$ and let us consider the functions $x=u+\dot u$ and $y=-u+\dot u$ (thus $u$ equals $(x-y)/2$. Those function satisfy the system
\begin{align}
\dot x &= x + f 
\label{dot_x}\\
\dot y &= -y + f
\label{dot_y}
\end{align}
Equation \cref{dot_y} shows that $y(t)$ approaches $0$ when $t$ approaches $+\infty$, and since the same assertion holds for $u(t)$, it must also hold for $x(t)$. Thus, according to equation \cref{dot_x}, the function $x(\cdot)$ must be given by:
\[
x(t) = -\int_t^{+\infty} e^{t-s} f(s) \, ds
\,.
\]
This provides an explicit expression for $x(0)$, thus also for $y(0)$ since $u(0)$ equals $0$, and finally for $y(t)$ for every nonnegative quantity $t$ according to \cref{dot_y}. It follows that $u(t)$ must be equal to the following expression for every nonnegative time $t$:
\begin{equation}
\label{expression_bounded_solution}
\frac{1}{2}\biggl( e^{-t} \int_0^t f(s) (e^{-s} - e^s) \, ds + (e^{-t} - e^t) \int_t^{+\infty} f(s) e^{-s} \, ds \biggr) 
\,,
\end{equation}
and this proves the uniqueness of a solution satisfying the conclusions \cref{prop_sol_forced_2nd_order} of \cref{lem:sol_forced_2nd_order}. Conversely, expression \cref{expression_bounded_solution} is the expression of a solution of equation \cref{equ_forced_2nd_order} and satisfies \cref{prop_sol_forced_2nd_order}. \Cref{lem:sol_forced_2nd_order} is proved. 
\end{proof}
\subsection{Example of two stable equilibria connected by two fronts travelling in opposite directions}
\label{subsec:ex_2_eq}
Let us consider the following reaction-diffusion equation (a small perturbation of the real Ginzburg-Landau equation): 
\begin{equation}
\label{gl}
A_t=A-\abs{A}^2A+\varepsilon^2 (\bar A+i\Omega A)+A_{xx}
\end{equation}
where the amplitude $A$ is complex, $\varepsilon$ is a small real quantity, and $\Omega$ is a real quantity in $(-1,1)$. This equation has been studied by P.~Coullet and J.-M.~Gilli as a model for nematic liquid crystals submitted to exterior electric and magnetic fields \cite{Coullet_locPattFronts_2002}. In polar coordinates $A=\rho e^{i\theta}$ this equation transforms into the following system:
\begin{equation}
\label{gl_polar_coord}
\left\{
\begin{aligned}
\partial_t\rho &= \rho-\rho^3+\varepsilon^2\rho\cos 2\theta +\partial_{xx}\rho-\rho\partial_{x}\theta^2 \\ 
\partial_t\theta &= \varepsilon^2(-\sin 2\theta +\Omega)+\frac{2\partial_{x}\rho\partial_{x}\theta}{\rho}+\partial_{xx}\theta
\end{aligned}
\right.
\end{equation}
The dynamics of the reaction system (without space) can be easily understood since the expression of $\partial_t\theta$ does not depend on $\rho$ (see \cref{fig:ph_gl_pert}). It has four equilibrium points close to the circle $\rho=1$:
\begin{itemize}
\item $\theta=(1/2)\arcsin\Omega$ and $\theta=\pi+(1/2)\arcsin\Omega$, those are stable, 
\item and $\theta=\pi/2-(1/2)\arcsin\Omega$ and $\theta=3\pi/2-(1/2)\arcsin\Omega$, those are saddles.
\end{itemize}
\begin{figure}[!htbp]
	\centering
    \includegraphics[width=0.8\textwidth]{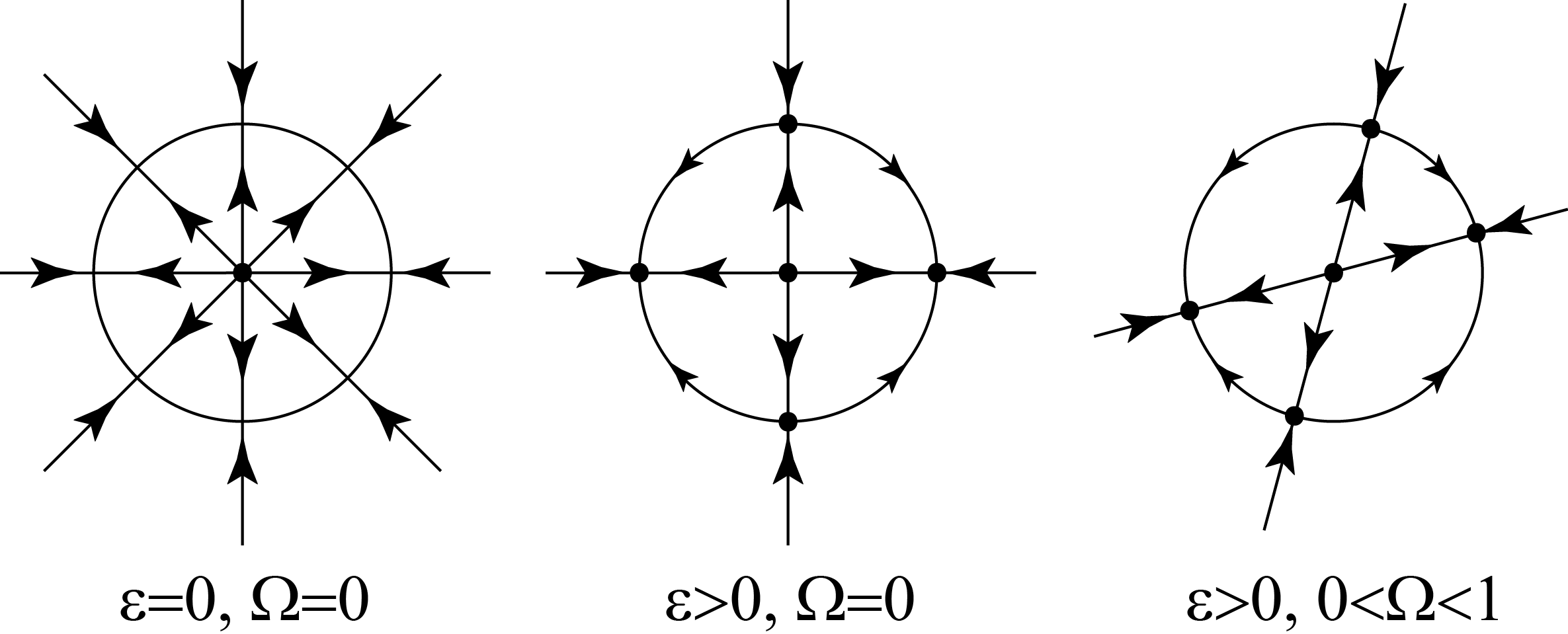}
    \caption{Phase space of the reaction equation.}
    \label{fig:ph_gl_pert}
\end{figure}
Using a perturbation argument, we are going to show that, for $\varepsilon$ close to $0$, the two stable equilibria are connected by two fronts travelling in opposite directions.

Let $c$ denote a real quantity. A front travelling at speed $c$ is a solution of system \cref{gl_polar_coord} of the form
\[
(x,t)\mapsto\bigl(\rho(x-ct),\theta(x-ct)\bigr)
\,.
\]
Replacing this ansatz into system \cref{gl_polar_coord} and performing the change of variables $\rho=1+\varepsilon^2 r$ yields the following system:
\[
\left\{
\begin{aligned}
-cr' &= -2r-3\varepsilon^2 r^2-\varepsilon^4r^3+(1+\varepsilon^2 r)\cos 2\theta+r''-\frac{1}{\varepsilon^2} (1+\varepsilon r)\theta'^2 \\ 
-c\theta' &= \varepsilon^2(-\sin 2\theta+\Omega)+\frac{2\varepsilon^2 r'\theta'}{1+\varepsilon^2 r}+\theta'' 
\end{aligned}
\right.
\]
Let us use the notation:
\[
\tilde r=r'
\quad\mbox{and}\quad
\tilde\theta=\frac{\theta'}{\varepsilon}
\quad\mbox{and}\quad
\tilde c =\frac{c}{\varepsilon}
\,,
\]
and let us consider the quantity $\tilde c$ not as a parameter, but as a (stationary) component of the differential system. The previous system transforms into the following first-order system:
\begin{equation}
\label{syst_fr_gl}
\left\{
\begin{aligned}
r' &= \tilde r \\ 
\tilde r' &= 2r -\cos2\theta+\tilde\theta^2
+\varepsilon\bigl(-\tilde{c}\tilde{r}+ 3\varepsilon r^2 + \varepsilon^3r^3-\varepsilon r\cos2\theta+\varepsilon r\tilde\theta^2 \bigr) \\ 
\theta' &= \varepsilon\tilde\theta \\ \
\tilde\theta' &= \varepsilon(-\tilde c \tilde\theta+\sin 2\theta-\Omega)-\varepsilon^2\frac{2 \tilde r\tilde\theta}{1+\varepsilon^2 r} \\ 
\tilde c' &= 0 
\end{aligned}
\right.
\end{equation}
At the limit $\varepsilon=0$, we get the ``fast'' system: 
\[
\left\{
\begin{aligned}
r' &= \tilde r \\ 
\tilde r' &= 2r -\cos2\theta+\tilde\theta^2 \\ 
\theta' &= 0 \\ 
\tilde\theta' &= 0 \\ 
\tilde c' &= 0 
\end{aligned}
\right.
\]
for which the graph $\Sigma_0$ of the map 
\[
H_0:\rr^3\rightarrow\rr^2, 
\quad
(\theta,\tilde\theta,\tilde c)\mapsto(r,\tilde r)=\Bigl(\frac{\cos2\theta-\tilde\theta^2}{2},0\Bigr)
\]
consists entirely of equilibrium points. At every point of $\Sigma_0$ this fast system is hyperbolic transversely to $\Sigma_0$; indeed, the eigenvalues of its differential are:
\begin{itemize}
\item $-\sqrt{2}$ and $\sqrt{2}$ (transversely to $\Sigma_0$),
\item and zero, with multiplicity three (in the direction of the tangent space to $\Sigma_0$).
\end{itemize}

We may thus apply Fenichel's global center manifold theorem \cite{Fenichel_geomSingPert_1979,Jones_geometricSingularPerturbationTheory_1995,Kaper_introductionGeometricMethodsSingularPerturbation_1999}. Let $D$ denote a compact and simply connected domain of $\rr^3$ with a smooth boundary (the choice of $D$ will be made later). According to this theorem, for $\varepsilon$ sufficiently close to zero, there exists a map $H_\varepsilon: D\rightarrow\rr^2$ that coincides with $H_0$ when $\varepsilon$ equals $0$ and depends smoothly on $\varepsilon$, and such that its graph $\Sigma_\varepsilon$ is locally invariant under the dynamics of~\cref{syst_fr_gl}. The dynamics on this ``slow'' manifold $\Sigma_\varepsilon$ thus reduces to the autonomous system
\begin{equation}
\label{front_gl_slow}
\left\{
\begin{aligned}
\ddot\theta &= -\tilde c\dot\theta+\sin 2\theta -\Omega - \varepsilon\frac{2\tilde{r}\dot\theta}{1+\varepsilon^2 r} \\ 
\dot{\tilde c} &= 0 
\end{aligned}
\right.
\end{equation}
where:
\begin{itemize}
\item derivatives are taken with respect to the ``slow'' time variable $y=\sqrt{\varepsilon}x$, namely:
\[
\dot\theta = \tilde{\theta} = \frac{\theta'}{\varepsilon}
\quad\mbox{and}\quad
\ddot\theta = \frac{\tilde\theta'}{\varepsilon} = \frac{\theta''}{\varepsilon^2}
\,;
\]
\item the quantities $r$ and $\tilde{r}$ are given by: $(r,\tilde{r}) = H_\varepsilon(\theta,\dot\theta,\tilde{c})$.
\end{itemize}
The first equation of system \cref{front_gl_slow} is a small perturbation of the dissipative oscillator 
\begin{equation}
\label{f_gl_slow_lim}
\ddot\theta=-\tilde c\dot\theta - V'(\theta)
\quad\mbox{where}\quad
V(\theta) = \frac{\cos2\theta}{2}+\Omega\,\theta\,,
\end{equation}
see \cref{fig:fr_gl}. 
\begin{figure}[!htbp]
	\centering
    \includegraphics[width=0.6\textwidth]{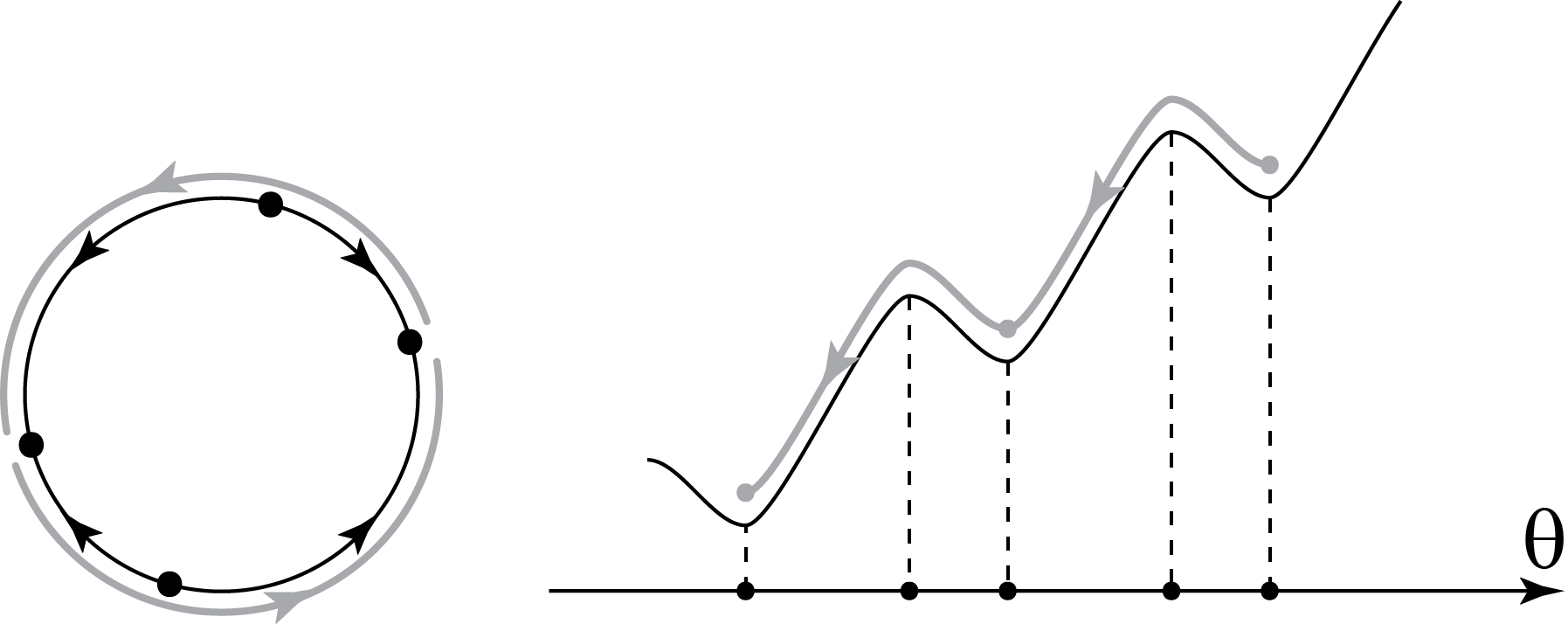}
    \caption{Bistable fronts travelling in opposite directions for equation~\cref{gl} ($\varepsilon>0$, $0<\Omega<1$).}
    \label{fig:fr_gl}
\end{figure}
Since $\Omega$ is in $(-1,1)$, the potential $V$ admits local maxima and minima, with periodicity $\pi$. Let us assume that $\Omega$ is nonzero, and let us consider the two successive local minima 
\[
\frac{\arcsin\Omega}{2}
\quad\mbox{and}\quad
\frac{\arcsin\Omega}{2}+\pi
\]
of the potential $V$. It is well known that there exists a unique nonzero quantity $\tilde c_0$ (depending on $\Omega$) such that, if $\tilde c$ equals $\tilde c_0$, these two minima are connected by a (unique) heteroclinic solution of the dissipative oscillator~\cref{f_gl_slow_lim}. It is also well-known that the corresponding travelling front for the reaction-diffusion equation 
\[
\theta_t=-V'(\theta)+\theta_{xx}
\]
is stable, and therefore robust with respect to small perturbations. In other words, for every quantity $\varepsilon$ sufficiently close to $0$, there exists a unique quantity $\tilde c$ close to $\tilde c_0$ such that the perturbed system \cref{front_gl_slow} admits a heteroclinic solution close to the previous one. To this heteroclinic solution corresponds a heteroclinic solution for the full system~\cref{syst_fr_gl} (provided that the domain $D$ was chosen large enough), and finally a travelling front for the initial equation~\cref{gl}, connecting the corresponding two stable equilibria of this initial equation. The same argument can be repeated for the local minima
\[
\frac{\arcsin\Omega}{2}+\pi
\quad\mbox{and}\quad
\frac{\arcsin\Omega}{2}+2\pi
\]
and proves the existence of the desired fronts travelling in opposite directions for initial equation~\cref{gl}.
\subsection{Isolation and robustness of the travelling front}
\label{subsec:robustness}
This \namecref{subsec:robustness} is devoted to the proof of \vref{prop:robustness}. All the arguments are standard, and we refer for instance to \cite{CoulletRieraTresser_stableStaticLocStructOneDim_2000,Coullet_locPattFronts_2002,
Sandstede_stabilityTW_2002,HomburgSandstede_homocHeteroclinicBifVectFields_2010,GuckenheimerKrauskopf_invManifGlobalBif_2015} for more details.

We keep the hypotheses and notation of \cref{subsec:setup,subsec:transv_assump}, except that the speed of the travelling front $\phi$ under consideration will be denoted by $c_0$ (instead of $c$ in \cref{subsec:setup,subsec:transv_assump}). The reason for this change is that it will be required below to consider a range of values for this speed (and not only the speed of the travelling front $\phi$). 
\subsubsection{Time stability (at both ends of the front) yields spatial hyperbolicity}
Steady states of system \cref{react_diff_trav_frame} --- that is, profiles of waves travelling at velocity $c$ --- are solutions of the system
\begin{equation}
\label{syst_trav_wave_c}
c v_\xi + F(v) + \ddd v_{\xi\xi} = 0 
\ \Longleftrightarrow \ 
v_{\xi\xi} = - \ddd^{-1} \bigl( F(v) + c v_\xi \bigr)
\,;
\end{equation}
the profile $\xi\mapsto \phi(\xi)$ of the travelling front is a solution of this system (which is identical to system \vref{syst_front}) for $c=c_0$. System \cref{syst_trav_wave_c} can be rewritten as the first order system:
\begin{equation}
\label{syst_trav_wave_c_first_order}
\left\{
\begin{aligned}
v' & = w \\
w' & = - \ddd^{-1} \bigl( F(v) + c w \bigr) 
\end{aligned}
\right.
\end{equation}
The following statement follows from hypothesis \hypStabInfty.
\begin{lemma}[system governing the profile of the front is hyperbolic at infinity]
\label{lem:spatially_hyperbolic}
Both \\
equilibria $(E_-,0)$ and $(E_+,0)$ of system \cref{syst_trav_wave_c_first_order} are hyperbolic, and their stable and unstable manifolds are $n$-dimensional.
\end{lemma}
\begin{proof}
The linearisation of system \cref{syst_trav_wave_c} at $E_-$ or $E_+$ and the linearisation of system \cref{syst_trav_wave_c_first_order} at $(E_-,0)$ and $(E_+,0)$ read (writing $E_\pm$ for $E_+$ or $E_-$):
\begin{equation}
\label{lin_syst_trav_wave_c_ends}
c v_\xi + DF(E_\pm)v + \ddd v_{\xi\xi} = 0
\ \Longleftrightarrow \ 
\left\{
\begin{aligned}
v' & = w \\
w' &= - \ddd^{-1} \bigl[ DF(E_\pm) v + c w \bigr]
\end{aligned}
\right.
\end{equation}
A complex quantity $\lambda$ is an eigenvalue of this linear system if and only if there exists a pair $(v,w)$ of vectors of $\cc^n$ such that:
\begin{equation}
\label{lin_syst_trav_wave_c_ends_eigenv}
\bigl[DF(E_\pm)v + \lambda^2 \ddd\bigr] v = - \lambda c v
\ \Longleftrightarrow \
\left\{
\begin{aligned}
\lambda v & = w \\
\lambda w &= - \ddd^{-1} \bigl[ DF(E_\pm) v + c w \bigr]
\end{aligned}
\right.
\end{equation}
It follows from hypothesis \hypStabInfty that such an eigenvalue $\lambda$ cannot be purely imaginary. Indeed, if we had $\lambda = ik$ for a real quantity $k$, then the last equation would read
\[
\bigl[DF(E_\pm)v - k^2 \ddd\bigr] v = - i k c v
\,,
\]
a contradiction with hypothesis \hypStabInfty (the spatially homogeneous equilibria $E_-$ and $E_+$ are assumed to be spectrally stable). This proves the hyperbolicity of $(E_\pm,0)$. If $c$ equals $0$ the solutions of the eigenvalue problem \cref{lin_syst_trav_wave_c_ends_eigenv} clearly go by pair of opposite complex numbers (this can also be viewed as a consequence of the space reversibility symmetry), thus in this case the dimensions of the stable and unstable manifolds are equal to $n$. Since the eigenvalues cannot cross the imaginary axis those dimensions remain equal to $n$ by continuity for every real quantity $c$. \Cref{lem:spatially_hyperbolic} is proved.
\end{proof}
\subsubsection{Algebraic multiplicity \texorpdfstring{$1$}{1} of the eigenvalue zero is equivalent to the transversality of the heteroclinic connection defining the profile of the front}
For every real quantity $c$, let 
\begin{itemize}
\item $W^{\textrm{u}}_c(E_-,0)$ denote the unstable manifold of the equilibrium $(E_-,0)$ and 
\item $W^{\textrm{s}}_c(E_+,0)$ denote the stable manifold of the equilibrium $(E_+,0)$
\end{itemize}
for system \cref{syst_trav_wave_c_first_order} (note that the subscript ``$c$'' refers to the speed of the travelling frame, \emph{not} to the concept of center manifold~!). According to \cref{lem:spatially_hyperbolic} above these manifolds $W^{\textrm{u}}_c(E_-,0)$ and $W^{\textrm{s}}_c(E_+,0)$ are $n$-dimensional submanifolds of $\rr^{2n}$. 
Now let us rewrite system \cref{syst_trav_wave_c_first_order} as a $2n+1$-dimensional system, with the speed $c$ as a variable instead of a parameter:
\begin{equation}
\label{syst_trav_wave_c_first_order_2n_plus_one}
\left\{
\begin{aligned}
v' & = w \\
w' & = - \ddd^{-1} \bigl( F(v) + c w \bigr) \\
c' &= 0
\end{aligned}
\right.
\Longleftrightarrow
\left\{
\begin{aligned}
cv' + F(v) + \ddd v'' &=0 \\
c' &= 0
\end{aligned}
\right.
\end{equation}
The flow in $\rr^{2n+1}$ of this system admits:
\begin{itemize}
\item a family of equilibria $\bigl\{(E_-,0,c):c\in\rr\bigr\}$ with an unstable manifold 
\[
\overline{W}^{\textrm{u}}\bigl((E_-,0)\times\rr\bigr) = \bigcup_{c\in\rr} \Bigl( W^{\textrm{u}}_{c}(E_-,0)\times \{c\} \Bigr)
\,;
\]
\item a family of equilibria $\bigl\{(E_+,0,c):c\in\rr\bigr\}$ with a stable manifold 
\[
\overline{W}^{\textrm{s}}\bigl((E_+,0)\times\rr\bigr) = \bigcup_{c\in\rr} \Bigl( W^{\textrm{s}}_{c}(E_+,0)\times \{c\} \Bigr)
\]
\end{itemize} 
(both are $n+1$-dimensional submanifolds of $\rr^{2n+1}$). 
The function
\[
\xi \mapsto \bigl( \phi(\xi), \phi'(\xi),c_0\bigr)
\]
is a solution of system \cref{syst_trav_wave_c_first_order_2n_plus_one} and its trajectory belongs to the intersection
\[
\overline{W}^{\textrm{u}}\bigl((E_-,0)\times\rr\bigr)
\cap
\overline{W}^{\textrm{s}}\bigl((E_+,0)\times\rr\bigr)
\,.
\]
Let us denote by $\Phi$ this trajectory (it is a subset of $\rr^{2n+1}$).
\begin{definition}[transverse travelling front]
The travelling front $\phi$ is said to be \emph{transverse} if the manifolds $\overline{W}^{\textrm{u}}\bigl((E_-,0)\times\rr\bigr)$ and $\overline{W}^{\textrm{s}}\bigl((E_+,0)\times\rr\bigr)$ intersect transversely along the trajectory $\Phi$. 
\end{definition}
\begin{lemma}[multiplicity of eigenvalue zero and transversality]
\label{lem:alg_mult_transv}
The travelling front $\phi$ is transverse if and only if the eigenvalue $0$ of the linearised operator
\[
\mathcal{L}:c_0\partial_\xi+ DF(\phi) + \ddd\partial_{\xi\xi}
\]
has an algebraic multiplicity equal to $1$.  
\end{lemma}
In other words, hypothesis \hypTransv is equivalent to the transversality of the travelling front $\phi$.
\begin{proof}
A small perturbation
\[
\xi \mapsto \bigl( \phi(\xi), \phi'(\xi),c_0\bigr) + \varepsilon \bigl( v(\xi), w(\xi), c(\xi) \bigr)
\]
is (at first order in $\varepsilon$) a solution of system \cref{syst_trav_wave_c_first_order_2n_plus_one} if $(v,w,c)$ are a solution of the linearised system:
\begin{equation}
\label{syst_trav_wave_c_first_order_2n_plus_one_lin}
\left\{
\begin{aligned}
v' & = w \\
w' & = - \ddd^{-1} \bigl( DF(\phi)v + c_0 w + c\phi'\bigr) \\
c' &= 0
\end{aligned}
\right.
\Longleftrightarrow
\left\{
\begin{aligned}
c_0 v' + DF(\phi)v + \ddd v'' &= - c\phi' \\
c' &= 0
\end{aligned}
\right.
\end{equation}
Observe that the restriction of system \cref{syst_trav_wave_c_first_order_2n_plus_one_lin} to the $2n$ first coordinates reads:
\[
\mathcal{L} v = - c \phi'
\,.
\]
The tangent space in $\rr^{2n+1}$ to the unstable manifold $\overline{W}^{\textrm{u}}\bigl((E_-,0)\times\rr\bigr)$ along $\Phi$ is made of the solutions of system \cref{syst_trav_wave_c_first_order_2n_plus_one_lin} satisfying 
\[
\bigl(v(\xi),w(\xi)\bigr)\rightarrow (0,0) 
\quad\mbox{when}\quad
\xi\rightarrow -\infty
\,,
\]
and the tangent space in $\rr^{2n+1}$ to the stable manifold $\overline{W}^{\textrm{s}}\bigl((E_+,0)\times\rr\bigr)$ along $\Phi$ is made of the solutions of system \cref{syst_trav_wave_c_first_order_2n_plus_one_lin} satisfying 
\[
\bigl(v(\xi),w(\xi)\bigr)\rightarrow (0,0) 
\quad\mbox{when}\quad
\xi\rightarrow +\infty
\,.
\]
According to \cref{lem:spatially_hyperbolic} these two tangent spaces are $n+1$-dimensional; besides their intersection contains (at least) the one-dimensional space $\spanset(\phi',0)$. Thus they intersect transversely if and only if their intersection is actually \emph{reduced} to $\spanset(\phi',0)$. And this is true if and only if there does not exist a quantity $c$ such that system \cref{syst_trav_wave_c_first_order_2n_plus_one_lin} admits a solution $\xi\mapsto v(\xi)$ outside of $\spanset(\phi')$ approaching zero at infinity; in other words, if and only if the eigenvalue $0$ of the operator $\mathcal{L}$ has algebraic multiplicity $1$. \Cref{lem:alg_mult_transv} is proved. 
\end{proof}
Since stable and unstable manifolds depend continuously on the reaction function $F$ and the diffusion matrix $\ddd$ defining system \cref{react_diff}, a transverse travelling front is isolated and robust (according to the definitions stated in \vref{subsec:transv_assump}). As a consequence, \vref{prop:robustness} follows from \cref{lem:alg_mult_transv}. \Cref{prop:robustness} is proved.
\begin{remark}
It can be seen from the proof of \cref{lem:alg_mult_transv} above that the null space of $\mathcal{L}$ is one-dimensional (that is, the eigenvalue zero has geometric multiplicity one) if and only if the intersection of $W^{\textrm{u}}_{c_0}(E_-,0)$ and $W^{\textrm{s}}_{c_0}(E_+,0)$ (in $\rr^{2n}$, without the additional dimension of the speed $c$) is transverse. And in this case, the algebraic multiplicity will also be one if and only if the Melnikov integral defined by the first order dependence of system \cref{syst_trav_wave_c_first_order} with respect to the parameter $c$ is nonzero \cite{GuckenheimerKrauskopf_invManifGlobalBif_2015}.
\end{remark}

It is commonly accepted that hypotheses \hypStabInfty and \hypTransv hold generically for travelling fronts of system \cref{react_diff} (say for a generic reaction function $F$ once the diffusion matrix $\ddd$ is fixed, or for a generic pair $(F,\ddd)$. Genericity of \hypStabInfty is standard since it reduces to the hyperbolicity of the equilibrium points of $F$. Concerning the second hypothesis \hypTransv, a rough justification follows from the equivalence with the transversality of the front. However, to the knowledge of the author, a rigorous justification of this transversality has not been written yet. A joint work in progress with Romain Joly aims at providing such a rigorous justification, however only under the additional hypothesis that the reaction term is the gradient of a potential. 
\subsubsection*{Acknowledgements} 
I am indebted to Régis Ferrière for introducing me to population dynamics, for asking me the question at the origin of this paper, and for many fruitful discussions. I am indebted to Gérard Iooss for fruitful discussions and his kind help, and I am grateful to the referees for their numerous and constructive remarks. In particular, it was a referee suggestion to apply the initial computation to the Lotka--Volterra competition model. 
%
%
\printbibliography 
%
%
\end{document}